\DeclareSymbolFont{AMSb}{U}{msb}{m}{n}
\documentclass[11pt,a4paper,twoside,leqno,noamsfonts]{amsart}

\usepackage[english]{babel}
\usepackage[dvipsnames]{xcolor}
\definecolor{britishracinggreen}{rgb}{0.0, 0.26, 0.15}
\definecolor{cobalt}{rgb}{0.0, 0.28, 0.67}
\usepackage[utopia]{mathdesign}
    \DeclareSymbolFont{usualmathcal}{OMS}{cmsy}{m}{n}
    \DeclareSymbolFontAlphabet{\mathcal}{usualmathcal}
\usepackage[a4paper,top=3.5cm,bottom=3cm,left=3cm,right=3cm,bindingoffset=5mm]{geometry}
\usepackage[utf8]{inputenc}
\usepackage{braket,caption,comment,mathtools}
\usepackage{multirow,booktabs,microtype}
\usepackage{todonotes}
\usepackage[colorlinks,bookmarks]{hyperref} %
\hypersetup{colorlinks,%
            citecolor=britishracinggreen,%
            filecolor=black,%
            linkcolor=cobalt,%
            urlcolor=black}
\numberwithin{equation}{section}


\def\be{\begin{equation}}    
\def\ee{\end{equation}}
\def\bitem{\begin{itemize}}
\def\eitem{\end{itemize}}
\def\benum{\begin{enumerate}}
\def\eenum{\end{enumerate}}

\def\into{\hookrightarrow}
\def\onto{\twoheadrightarrow}

\def\op{\textrm{op}}
\def\isom{\cong}  
\def\ra{\rightarrow}

\def\surj{\twoheadrightarrow}
\def\Var{\mathsf{Var}}
\def\Sch{\mathsf{Sch}}
\def\Sets{\mathsf{Sets}}
\def\St{\mathsf{St}}

\def\L{\mathbb L}
\def\A{\mathbb A}
\def\R{\mathbb R}
\def\C{\mathbb C}

\def\P{\mathbb P}
\def\Q{\mathbb Q}
\def\G{\mathbb G}
\def\L{\mathbb{L}}

\def\Z{\mathbb Z}
\def\ext{\mathrm{ext}}

\def\O{\mathscr O}
\def\DDT{\mathsf{DT}}
\def\PPT{\mathsf{PT}}

\def\sc{\textrm{sc}}

\def\Ad{\textrm{Ad}}
\def\reg{\textrm{reg}}

\DeclareMathOperator{\Quot}{Quot}
\DeclareMathOperator{\Hilb}{Hilb}
\DeclareMathOperator{\Chow}{Chow}

\DeclareMathOperator{\ch}{ch}
\DeclareMathOperator{\Td}{Td}

\DeclareMathOperator{\id}{id}
\DeclareMathOperator{\Pic}{Pic}

\DeclareMathOperator{\Rep}{Rep}

\DeclareMathOperator{\vir}{vir}
\DeclareMathOperator{\QCoh}{QCoh}
\DeclareMathOperator{\Coh}{Coh}

\DeclareMathOperator{\Spec}{Spec\,}

\DeclareMathOperator{\Supp}{Supp\,}
\DeclareMathOperator{\coker}{coker}

\DeclareMathOperator{\Perf}{Perf}

\DeclareMathOperator{\DT}{DT}
\DeclareMathOperator{\PT}{PT}

\DeclareMathOperator{\GL}{GL}

\DeclareMathOperator{\dd}{d}
\DeclareMathOperator{\Tr}{Tr}

\DeclareMathOperator{\Sym}{Sym}

\DeclareMathOperator{\Ext}{Ext}
\DeclareMathOperator{\lExt}{{\mathscr Ext}}
\DeclareMathOperator{\Hom}{Hom}
\DeclareMathOperator{\lHom}{{\mathscr Hom}}
\DeclareMathOperator{\catA}{{\mathscr A}}
\DeclareMathOperator{\catB}{{\mathscr B}}
\DeclareMathOperator{\catC}{{\mathcal C}}

\DeclareMathOperator{\catT}{{\mathscr T}}
\DeclareMathOperator{\catF}{{\mathscr F}}
\DeclareMathOperator{\End}{End}

\DeclareMathOperator{\rk}{rk}

\DeclareMathOperator{\pr}{pr}

\DeclareMathOperator{\F}{\mathcal F}

\theoremstyle{definition}

\newtheorem*{lemma*}{Lemma}
\newtheorem*{theorem*}{Theorem}
\newtheorem*{example*}{Example}
\newtheorem*{fact*}{Fact}
\newtheorem*{notation*}{Notation}
\newtheorem*{definition*}{Definition}
\newtheorem*{prop*}{Proposition}
\newtheorem*{remark*}{Remark}
\newtheorem*{corollary*}{Corollary}

\newtheorem{definition}{Definition}[section]

\newtheorem{example}[definition]{Example}

\newtheorem{question}[definition]{Question}
\newtheorem{remark}[definition]{Remark}

\newtheoremstyle{thm} 
        {3mm}
        {3mm}
        {\slshape}
        {0mm}
        {\bfseries}
        {.}
        {1mm}
        {}
\theoremstyle{thm}
\newtheorem{theorem}[definition]{Theorem}
\newtheorem{corollary}[definition]{Corollary}
\newtheorem{lemma}[definition]{Lemma}
\newtheorem{prop}[definition]{Proposition}
\newtheorem{thm}{Theorem}

\usepackage{tikz}
\usepackage{tikz-cd}
\usepackage{rotating}
\newcommand*{\isoarrow}[1]{\arrow[#1,"\rotatebox{90}{\(\sim\)}"
]}
\usetikzlibrary{matrix,shapes,arrows,decorations.pathmorphing}
\tikzset{commutative diagrams/arrow style=math font}
\tikzset{commutative diagrams/.cd,
mysymbol/.style={start anchor=center,end anchor=center,draw=none}}
\newcommand\MySymb[2][\square]{%
  \arrow[mysymbol]{#2}[description]{#1}}
\tikzset{
shift up/.style={
to path={([yshift=#1]\tikztostart.east) -- ([yshift=#1]\tikztotarget.west) \tikztonodes}
}
}

\DeclareMathAlphabet{\mathpzc}{OT1}{pzc}{m}{it}

\newcommand*{\defeq}{\mathrel{\vcenter{\baselineskip0.5ex \lineskiplimit0pt
                     \hbox{\scriptsize.}\hbox{\scriptsize.}}}%
                     =}

\title[Virtual counts on Quot schemes and DT/PT wall-crossing]{Virtual counts on Quot schemes \\ and the higher rank local DT/PT correspondence}

\author[S. V. Beentjes]{Sjoerd V. Beentjes}
\address{School of Mathematics and Maxwell Institute,
University of Edinburgh,
James Clerk Maxwell Building,
Peter Guthrie Tait Road, Edinburgh, EH9 3FD,
United Kingdom}
\email[Sjoerd Beentjes]{sjoerd.beentjes@ed.ac.uk}

\author[A. T. Ricolfi]{Andrea T. Ricolfi}
\address{SISSA Trieste,
Via Bonomea 265, 
34136 Trieste,
Italy}
\email[Andrea Ricolfi]{aricolfi@sissa.it}

\begin{document}

\begin{abstract}
We show that the Quot scheme $\Quot_{\A^3}(\O^r,n)$ admits a symmetric obstruction theory, and we compute its virtual Euler characteristic. We extend the calculation to locally free sheaves on smooth $3$-folds, thus refining a special case of a recent Euler characteristic calculation of Gholampour--Kool. We then extend Toda's higher rank DT/PT correspondence on Calabi--Yau $3$-folds to a local version centered at a fixed slope stable sheaf. This generalises (and refines) the local DT/PT correspondence around the cycle of a Cohen--Macaulay curve. Our approach clarifies the relation between Gholampour--Kool's functional equation for Quot schemes, and Toda's higher rank DT/PT correspondence.
\end{abstract}

\maketitle

{\hypersetup{linkcolor=black}
\tableofcontents}

\section{Introduction}

Let $X$ be a smooth projective Calabi--Yau $3$-fold over $\C$.
Donaldson--Thomas (DT) invariants, introduced in \cite{ThomasThesis}, are virtual counts of stable objects in the bounded derived category $D^b(X)$ of $X$.
Particularly well-studied examples of such stable objects are \emph{ideal sheaves} (rank one torsion free sheaves with trivial determinant) of curves in $X$ and the \emph{stable pairs} of Pandharipande--Thomas (PT) \cite{PT}.
These objects are related by wall-crossing phenomena (cf.~Section \ref{sec:previouswork}), giving rise to the famous DT/PT correspondence \cite{Bri,Toda1,Toda0}.

Recently Toda \cite{Toda2} generalised the classical (rank one) DT/PT correspondence to arbitrary rank.
Let $\omega$ be an ample class on $X$, and let $(r,D) \in H^0(X) \oplus H^2(X)$ be a pair such that $r\geq 1$ and $\gcd(r,D \cdot \omega^2)=1$.
The higher rank analogues of ideal sheaves are $\omega$-slope stable torsion free sheaves of Chern character $\alpha = (r,D,-\beta,-m)$, for given $(\beta,m) \in H^4(X) \oplus H^6(X)$.
The higher rank analogues of stable pairs are certain two-term complexes $J^\bullet \in D^b(X)$ of class $\alpha$, first described by Lo (see Section \ref{sec:Moduli}), called \emph{PT pairs} \cite{Lo}.
The virtual counts of these objects can be computed as Behrend's virtual Euler characteristic of their moduli space.

Toda's higher rank wall-crossing formula \cite[Thm.~1.2]{Toda2} is the equality
\be\label{Todaformula}
\DDT_{r,D,\beta}(q) = \mathsf M((-1)^rq)^{r\chi(X)}\cdot \PPT_{r,D,\beta}(q)
\ee
of generating series.
Here $\DDT_{r,D,\beta}$ is the generating function of DT invariants of Chern characters of the form $(r,D,-\beta,-m)$, with $m \in H^6(X)$ varying, and similarly for the series $\PPT_{r,D,\beta}$. 
The ``difference'' between the DT and PT generating functions is measured by a \emph{wall-crossing factor}, expressed in terms of the MacMahon function 
\[
\mathsf M(q) = \prod_{m\geq 1}(1-q^m)^{-m},
\]
the generating function of plane partitions of natural numbers.

In \cite{Gholampour2017}, Gholampour--Kool proved a formally similar relation in the following situation.
Fix a torsion free sheaf $\mathcal F$ of rank $r$ and homological dimension at most one on a smooth projective $3$-fold $X$, not necessarily Calabi--Yau.
Then \cite[Thm.~1.1]{Gholampour2017} states the equality
\be\label{GKformula}
\sum_{n\geq 0}\chi(\Quot_X(\mathcal F,n))q^n = \mathsf M(q)^{r\chi(X)}\cdot \sum_{n\geq 0}\chi(\Quot_X(\lExt^1(\mathcal F,\O_X),n))q^n,
\ee
where $\chi$ is the topological Euler characteristic and $\Quot_X(E,n)$ is the Quot scheme of length $n$ quotients of the coherent sheaf $E$.

In this paper, we explain the formal similarity between equations~\eqref{Todaformula} and \eqref{GKformula}, answering a question raised in \cite[Sec.~1]{Gholampour2017}. Our method exhibits formula \eqref{GKformula} as an Euler characteristic shadow of a \emph{local} version of~\eqref{Todaformula} based at the sheaf $\mathcal{F}$.
Moreover, we explicitly compute (cf.~Corollary \ref{cor:locfree8237}) the $\mathcal{F}$-local DT generating series when $\mathcal{F}$ is a locally free sheaf: this can be seen (cf.~Remark \ref{rmk:comparison}) as the higher rank analogue of the point contribution to rank one DT theory, originally the first of the three MNOP conjectures, cf.~\cite[Conj.~1]{MNOP1}.

\subsection{Main results}
We give some details and state our results in order of appearance. In the final part of the Introduction (Section \ref{sec:previouswork}) we give a short outline of related work on wall-crossing and Quot schemes in Enumerative Geometry. 

To any complex scheme $Y$ of finite type, Behrend \cite{Beh} associates a canonical constructible function $\nu_Y\colon Y(\C)\ra \Z$.
We recall in Section~\ref{behrendstuff} the properties we will need.
The \emph{virtual Euler characteristic} of $Y$ is the ``motivic integral'' of $\nu_Y$, namely
\be\label{virtualchi}
\widetilde\chi(Y) \defeq \chi(Y,\nu_Y) \defeq \sum_{k\in \Z}k\cdot \chi(\nu_Y^{-1}(k)) \in \Z.
\ee
Our first result, proven in Section~\ref{sec:thmA}, is the following virtual refinement of Gholampour--Kool's formula \eqref{GKformula} in the locally free case.
\begin{thm}\label{Thm:LocallyFree}
Let $X$ be a smooth $3$-fold, $\mathcal F$ a locally free sheaf of rank $r$ on $X$. Then
\begin{equation}\label{us}
\sum_{n\geq 0}\widetilde\chi(\Quot_X(\mathcal F,n)) q^n = \mathsf M((-1)^rq)^{r\chi(X)}.
\end{equation}
\end{thm}
The case $r = 1$, corresponding to the Hilbert scheme parametrising zero-dimensional subschemes of $X$, has been proven by Behrend--Fantechi in \cite{BFHilb}.
We establish \eqref{us} by generalising their approach (technical details are in Appendix \ref{quotmess}). See also \cite{LEPA,JLI} for different proofs for $r=1$. Note that no Calabi--Yau or projectivity assumptions on $X$ are required.

\begin{remark}
If $\mathcal F$ is a locally free sheaf, a local model for $\Quot_X(\mathcal F,n)$ is the Quot scheme $\Quot_{\A^3}(\O^r,n)$. We show that the latter is a \emph{critical locus} (Theorem \ref{thm:criticallocus}), so that in particular it carries a symmetric perfect obstruction theory in the sense of \cite{BFHilb}. This motivates the interest in the virtual Euler characteristic computed in Theorem \ref{Thm:LocallyFree}. However, even for \emph{reflexive} sheaves $\mathcal F$ over Calabi--Yau $3$-folds, we do not know when $\Quot_X(\mathcal F,n)$ carries a symmetric perfect obstruction theory.
\end{remark}

We now describe the local higher rank DT/PT correspondence.
For a given $\omega$-slope stable sheaf $\F$ of homological dimension at most one, we embed the Quot schemes
\[
\phi_{\F}\colon\Quot_X(\mathcal F)\hookrightarrow M_{\DT}(r,D),\qquad \psi_{\F}\colon \Quot_X(\lExt^1(\mathcal F,\O_X))\hookrightarrow M_{\PT}(r,D)
\]
in suitable DT and PT moduli spaces via closed immersions (cf.~Propositions~\ref{DTembedding} and~\ref{PTembedding}).
The former, $\phi_{\F}$, consists of taking the kernel of a surjection $\mathcal{F} \onto Q$.
The latter, $\psi_{\F}$, might be of independent interest, so we describe its action on $\C$-valued points here.

Let $t \colon \lExt^1(\F,\O_X) \onto Q$ be a zero-dimensional quotient.
Since $\F$ is of homological dimension at most one, there is a natural morphism
\[
\bar{t} \colon \F^{\vee} \to \lExt^1(\F,\O_X)[-1] \xrightarrow{t[-1]} Q[-1],
\]
where $(-)^{\vee} = {\textbf R}\lHom(-,\O_X)$ is the derived dualising functor.
This functor is an involutive anti-equivalence of $D^b(X)$, hence dualising again yields a canonical isomorphism
\[
\Hom(\F^{\vee},Q[-1]) \cong \Ext^1(Q^{\mathrm D}[-1],\F),
\]
where $Q^{\mathrm D} = \lExt^3(Q,\O_X)$.
We define $\psi_{\F}$ by sending $t$ to the corresponding extension
\[
\F \to J^{\bullet} \to Q^{\mathrm D}[-1]
\]
in $D^b(X)$.
We prove in Proposition~\ref{PTembedding} that this defines a higher rank PT pair and, moreover, that the association $t\mapsto J^\bullet$ extends to a morphism that is a closed immersion.

We define $\F$-\emph{local DT} and \emph{PT invariants} by restricting the Behrend weights on the full DT and PT moduli spaces via these closed immersions.\footnote{In general, these immersion are not open, so the restriction of the Behrend weight of the full moduli space does \emph{not} in general agree with the intrinsic Behrend weight of the Quot schemes.}
We collect these in a generating function
\[
\DDT_{\mathcal F}(q) = \sum_{n\geq 0}\chi\left(\Quot_X(\mathcal F,n),\nu_{\DT}\right) q^n
\]
on the DT side, and a similar generating function $\PPT_{\F}(q)$ on the PT side.

We prove the following relation (Theorem \ref{thm:DTPT}) after establishing a key identity in a certain motivic Hall algebra, and applying the Behrend weighted integration morphism.
\begin{thm}\label{thm2}
Let $X$ be a smooth projective Calabi--Yau $3$-fold, and let $\F$ be an $\omega$-stable torsion free sheaf on $X$ of rank $r$ and homological dimension at most one.
Then
\be\label{dtptF}
\DDT_{\mathcal F}(q) = \mathsf M((-1)^rq)^{r\chi(X)} \cdot \PPT_{\mathcal F}(q).
\ee
\end{thm}
Applying the unweighted integration morphism instead, taking Euler characteristics, we recover the main result of \cite{Gholampour2017} in the special case where $\F$ is slope-stable.
\begin{thm}\label{thm3}
Let $X$ be a smooth projective $3$-fold, and let $\F$ be an $\omega$-stable torsion free sheaf on $X$ of rank $r$ and homological dimension at most one.
Then Gholampour--Kool's formula \eqref{GKformula} holds.
\end{thm}
Since the formula of Gholampour--Kool descends from the same Hall algebra identity giving rise to \eqref{dtptF}, one can interpret it as an Euler characteristic shadow of the $\F$-local higher rank DT/PT correspondence, in the case that $\F$ is $\omega$-slope stable.

In the final section, we consider some special cases by imposing further restrictions on $\F$.
In particular, the following result is the `intersection' between Theorems \ref{Thm:LocallyFree} and \ref{thm2}.
\begin{corollary}\label{cor:locfree8237}
With the assumptions of Theorem~\ref{thm2}, if $\F$ is locally free then 
\[
\DDT_{\F}(q) = \mathsf M((-1)^rq)^{r\chi(X)}.
\]
\end{corollary}

We show in Proposition~\ref{prop:PT_tensor_by_line_bundle} the invariance property $\PPT_{\F \otimes L}(q) = \PPT_{\F}(q)$, where $L$ is a line bundle on $X$. Finally, Corollary \ref{cor292} shows that $\PPT_{\F}$ is a polynomial if $\F$ is reflexive.

\subsection{Previous work on wall-crossing}\label{sec:previouswork}
Both the DT invariant and the Euler characteristic
\[
\widetilde\chi(M_X(\alpha))\in \Z,\quad \chi(M_X(\alpha))\in \Z
\]
can be seen as ways to size, or ``count points'' in the moduli space
$M_X(\alpha)$ of stable torsion free sheaves of class $\alpha$.
Unlike the virtual invariants, the Euler characteristic is not \emph{deformation invariant}.
Indeed, deformation invariance of $\widetilde\chi(M_X(\alpha))$ is a consequence of the virtual class technology available for $M_X(\alpha)$.
Nonetheless, the virtual and the naive invariants share a common behavior: both satisfy \emph{wall-crossing formulas}.
These are relations describing the transformation that the invariants undergo when one deforms not the complex structure of $X$, but rather the \emph{stability condition} defining the moduli space.

The first calculations in DT theory, and their original motivation, involved ideal sheaves, namely the rank one DT objects; see for instance \cite{MNOP1,BFHilb}.

The DT/PT correspondence was first phrased rigorously as a wall-crossing phenomenon by Bayer \cite{Bayer1} using polynomial stability conditions. It was established in the rank one case by Bridgeland \cite{Bri} and Toda \cite{Toda1} for the virtual invariants, and previously by Toda \cite{Toda0} for Euler characteristics.

Similar wall-crossing formulas hold for the Quot schemes
\[
\Quot_X(\mathscr I_C,n)\subset \Hilb(X),
\]
where $C\subset X$ is a fixed Cohen--Macaulay curve and $X$ need not be Calabi--Yau.
See work of Stoppa--Thomas \cite{ST} for the naive invariants, and \cite{LocalDT} for the virtual ones (where also projectivity is not needed, but $C$ is required to be smooth).

In \cite{Ob1,Ricolfi2018}, a cycle-local DT/PT correspondence is proved in rank one.
Cycle-local DT invariants gather contributions of those ideals $\mathscr I_{Z}\subset \O_X$ for which the cycle $[Z]$ is equal to a fixed element of the Chow variety of $X$.
Specialising to the case $\F = \mathscr I_C$, where $C \subset X$ is a Cohen--Macaulay curve, our Theorem~\ref{thm2} refines this cycle-local DT/PT correspondence to an $\mathscr I_C$-local correspondence around the fixed subscheme $C\subset X$ (see Remark \ref{rmk:cyclelocal}).

In the recent work \cite{Lo2018}, Lo approaches the problem of relating $\Quot_X(\lExt^1(\F,\O_X))$ to the moduli space $M_{\PT}(r,D)$ from a categorical point of view, constructing a functor `the other way around', i.e.~from PT pairs to quotients.
It would be interesting to find out the precise relationship between Lo's functor and our closed immersion.

\subsection*{Acknowledgements}
We would like to thank Arend Bayer, Martijn Kool, Georg Oberdieck, and Richard Thomas for helpful comments and suggestions during the completion of this work. We would also like to thank the referee for a careful reading of the paper, and for corrections and suggestions. This project was set up during a visit of A.R. to the University of Edinburgh, and it was completed during a visit of S.B. to the Max-Planck Institute in Bonn.
We thank both institutions for their hospitality and excellent working conditions. S.B. was supported by the ERC Starting Grant no.~337039 \emph{WallXBirGeom}.

\subsection*{Conventions}
We work over $\C$.
All rings, schemes, and stacks will be assumed to be \emph{locally of finite type} over $\C$, unless specified otherwise. All categories and functors will be $\C$-linear. A Calabi--Yau $3$-fold is a smooth projective $3$-fold $X$ such that $K_X=0$ and $H^1(X,\O_X)=0$.
If $M$ is a scheme, we write $D^b(M)$ for the \emph{bounded coherent} derived category of $M$ and $\Perf(M)\subset D^b(M)$ for the category of perfect complexes.
We write $\Coh_0(M)$ (resp.~$\Coh_{\leq 1}(M)$) for the full subcategory of coherent sheaves on $M$ of dimension zero (resp.~of dimension at most $1$). The \emph{homological dimension} of a coherent sheaf $\F$ on a smooth projective variety is the minimal length of a locally free resolution of $\F$. For a sheaf $E\in \Coh(X)$ of codimension $c$, we set $E^{\mathrm D} = \lExt^c(E,\O_X)$.\footnote{This differs slightly from the notation used in \cite{modulisheaves}, where $E^{\mathrm D}$ denotes $\lExt^c(E,\omega_X)$.}
We write $\Quot_X(E,n)$ for the Quot scheme of length $n$ quotients of $E$, and we set 
$\Quot_X(E)=\coprod_{n\geq 0}\Quot_X(E,n)$.
We refer the reader to \cite{modulisheaves} for generalities on moduli spaces of sheaves and to \cite[Sec.~1]{PicScheme} for base change theory for local ext.

\section{Quotients of a free sheaf on affine 3-space}
B.~Szendr\H{o}i proved in \cite[Theorem 1.3.1]{MR2403807} that the Hilbert scheme of points $\Hilb^n\A^3$ is a global critical locus. The goal of this section is to prove that the same is true, more generally, for the Quot scheme
\[
\Quot_{\A^3}(\O^r,n),
\]
for all $r\geq 1$ and $n\geq 0$.  
In other words, we will show that the Quot scheme can be written as the scheme-theoretic zero locus of an exact one-form $\dd f$, where $f$ is a regular function defined on a smooth scheme. In particular, this proves that $\Quot_{\A^3}(\O^r,n)$ carries a symmetric perfect obstruction theory, defined by the Hessian of $f$. By inspecting the virtual \emph{motivic} refinements of these critical loci, we deduce the formula
\[
\sum_{n\geq 0}\widetilde\chi(\Quot_{\A^3}(\O^r,n)) q^n = \mathsf M((-1)^rq)^r.
\]
In Section~\ref{sec:thmA}, we generalise this formula to arbitrary locally free sheaves on smooth quasi-projective $3$-folds, thus establishing Theorem~\ref{Thm:LocallyFree}.

\subsection{Quiver representations}

Let $Q$ be a quiver, i.e.~a finite directed graph. We denote by $Q_0$ and $Q_1$ the sets of vertices and edges of $Q$ respectively.
A representation $M$ of $Q$ is the datum of a finite dimensional vector space $V_i$ for every $i\in Q_0$, and a linear map $V_i\ra V_j$ for every edge $i\ra j$ in $Q_1$.
The \emph{dimension vector} of $M$ is
\[
\underline{\dim}\, M = (\dim V_i)\in \mathbb N^{Q_0}.
\]
It is well known that the representations of $Q$ form an abelian category, that is moreover equivalent to the category of left modules over the path algebra $\C Q$ of the quiver.

Following \cite{King}, we recall the notion of (semi)stability of a representation of $Q$.

\begin{definition}\label{centralcharge}
A \emph{central charge} is a group homomorphism $Z\colon \mathbb Z^{Q_0}\ra \C$ such that the image of $\mathbb N^{Q_0}\setminus 0$ lies inside $\mathbb H_+ = \set{re^{i\pi\varphi}|r>0,\,0<\varphi\leq 1}$. For every $\alpha\in \mathbb N^{Q_0}\setminus 0$, we denote by $\varphi(\alpha)$ the real number $\varphi$ such that $Z(\alpha) = re^{i\pi\varphi}$. It is called the \emph{phase} of $\alpha$.
\end{definition}

Note that every vector $\theta\in \R^{Q_0}$ induces a central charge $Z_{\theta}$ given by
\[
Z_{\theta}(\alpha) = -\theta\cdot \alpha + i|\alpha|,
\]
where $|\alpha| = \sum_i\alpha_i$. We denote by $\varphi_\theta$ the induced phase function, and we put
\[
\varphi_\theta(M) = \varphi_\theta(\underline{\dim}\,M)
\]
for every $Q$-representation $M$.

\begin{definition}\label{stablereps}
Fix $\theta\in \R^{Q_0}$. Then a representation $M$ of $Q$ is called \emph{$\theta$-semistable} if 
\[
\varphi_\theta(A)\leq \varphi_\theta(M)
\]
for every nonzero proper subrepresentation $A\subset M$. When strict inequality holds, we say that $M$ is \emph{$\theta$-stable}. Vectors $\theta\in \R^{Q_0}$ are referred to as \emph{stability parameters}.
\end{definition}

Fix a stability parameter $\theta\in \R^{Q_0}$. To each $\alpha\in\mathbb N^{Q_0}\setminus 0$ one can associate its \emph{slope} (with respect to $\theta$), namely the rational number
\[
\mu_\theta(\alpha) = \frac{\theta\cdot\alpha}{|\alpha|}\in\Q.
\]
It is easy to see that $\varphi_\theta(\alpha)<\varphi_\theta(\beta)$ if and only if $\mu_\theta(\alpha)<\mu_\theta(\beta)$. So, after setting $\mu_\theta(M) = \mu_\theta(\underline{\dim}\,M)$, one can check stability using slopes instead of phases.

\subsection{Framed representations}
Let $Q$ be a quiver with a distinguished vertex $0\in Q_0$, and let $r$ be a positive integer. Consider the quiver $\widetilde Q$ obtained by adding one vertex $\infty$ to the original vertices in $Q_0$ and $r$ edges $\infty\ra 0$.
If $r = 1$, this construction is typically referred to as a \emph{framing} of $Q$. 

A representation $\widetilde M$ of $\widetilde Q$ can be uniquely written as a pair $(M,v)$, where $M$ is a representation of $Q$ and $v = (v_1,\dots,v_r)$ is an $r$-tuple of linear maps $v_i\colon V_\infty\ra V_0$. We always assume our framed representations to satisfy $\dim V_\infty = 1$, so that
\[
\underline{\dim}\,\widetilde M=(1,\underline{\dim}\,M).
\]
The vector space $V_{\infty}$ will be left implicit.

\begin{definition}\label{framedstability}
Let $\theta\in \mathbb R^{Q_0}$ be a stability parameter.
A representation $(M,v)$ of $\widetilde Q$ with $\dim V_\infty = 1$ is said to be \emph{$\theta$-(semi)stable} if it is $(\theta_\infty,\theta)$-(semi)stable in the sense of Definition \ref{stablereps}, where $\theta_\infty = -\theta\cdot \underline{\dim}\,M$.
\end{definition}

The space of all representations of $Q$, of a given dimension vector $\alpha\in \mathbb N^{Q_0}$, is the affine space
\[
\Rep_\alpha(Q) = \prod_{i\ra j}\Hom_\C(\C^{\alpha_i},\C^{\alpha_j}).
\]
On this space there is an action of the gauge group $\GL_\alpha = \prod_{i\in Q_0} \GL_{\alpha_i}$ by simultaneous conjugation. The quotient stack $\Rep_\alpha(Q)/\GL_\alpha$ parametrises isomorphism classes of representations of $Q$ with dimension vector $\alpha$. Imposing suitable stability conditions, one can restrict to $\GL_\alpha$-invariant open subschemes of $\Rep_\alpha(Q)$ such that the induced action is free, so that the quotient is a smooth \emph{scheme}. We will do this in the next subsection for framed representations of the three loop quiver.

\subsection{The three loop quiver}
Consider the quiver $\mathsf L_3$ with one vertex and three loops, labelled $x$, $y$ and $z$. A representation of $\mathsf L_3$ is the datum of a left module over the free algebra
\[
\C\langle x,y,z\rangle
\]
on three generators, which is the path algebra of $\mathsf L_3$.
We now add $r$ framings at the vertex, thus forming the quiver $\widetilde{\mathsf L}_{3}$ (see Fig.~\ref{L3quiver}).

\begin{figure}[ht]
\begin{tikzpicture}[>=stealth,->,shorten >=2pt,looseness=.5,auto]
  \matrix [matrix of math nodes,
           column sep={3cm,between origins},
           row sep={3cm,between origins},
           nodes={circle, draw, minimum size=7.5mm}]
{ 
|(A)| \infty & |(B)| 0 \\         
};
\tikzstyle{every node}=[font=\small\itshape]
\path[->] (B) edge [loop above] node {$x$} ()
              edge [loop right] node {$y$} ()
              edge [loop below] node {$z$} ();

\node [anchor=west,right] at (-0.2,0.1) {$\vdots$};
\node [anchor=west,right] at (-0.3,0.95) {$v_1$};              
\node [anchor=west,right] at (-0.3,-0.85) {$v_r$};              
\draw (A) to [bend left=25,looseness=1] (B) node [midway,above] {};
\draw (A) to [bend left=40,looseness=1] (B) node [midway] {};
\draw (A) to [bend right=35,looseness=1] (B) node [midway,below] {};
\end{tikzpicture}\caption{The quiver $\widetilde{\mathsf L}_3$ with its $r$ framings.}\label{L3quiver}
\end{figure}

Letting $V_n$ be a fixed $n$-dimensional vector space, representations of $\widetilde{\mathsf L}_{3}$ with dimension vector $(1,n)$ form the affine space
\[
\mathcal R_{n,r} = \End(V_n)^3\times V_n^r,
\]
of dimension $3n^2+rn$. Consider the open subset
\be\label{unr}
U_{n,r}\subset \mathcal R_{n,r}
\ee
parametrising tuples $(A,B,C,v_1,\dots,v_r)$ such that the vectors $v_1,\dots,v_r$ span the underlying (unframed) representation $(A,B,C)\in \Rep_n(\mathsf L_3)$. Equivalently, the $r$ vectors $v_i$ span $V_n$ as a $\C\langle x,y,z\rangle$-module.

Consider the action of $\GL_n$ on $\mathcal R_{n,r}$ given by 
\[
g\cdot (A,B,C,v_1,\dots,v_r) = (A^g,B^g,C^g,g v_1,\dots,g v_r),
\]
where $A^g$ denotes the conjugation $gAg^{-1}$ by $g\in \GL_n$. This action is free on $U_{n,r}$.
Thus the GIT quotient $U_{n,r}/\GL_n$ with respect to the character $\det \colon \GL_n \to \C^{\times}$ is a smooth quasi-projective variety.

We now show that the open set $U_{n,r}$ parametrises stable framed representations of $\mathsf L_3$, with stability understood in the sense of Definition \ref{framedstability}.

If $r=1$, then for a point $(A,B,C,v)\in U_{n,1}$ one usually says that $v$ is a \emph{cyclic vector}. For $r>1$, we say that $v_1,\dots,v_r$ \emph{jointly generate} a representation $M=(A,B,C)\in \Rep_n(\mathsf L_3)=\End(V_n)^3$ if $(M,v_1,\dots,v_r)\in U_{n,r}$. 

\begin{prop}
Let $\widetilde M = (M,v_1,\dots,v_r)$ be a representation of the framed quiver\, $\widetilde{\mathsf L}_{3}$ depicted in Figure \ref{L3quiver}.
	Choose a vector $\theta = (\theta_1,\theta_2)$ with $\theta_1 > \theta_2$. Then $\widetilde M$ is $\theta$-stable if and only if $v_1,\ldots,v_r$ jointly generate $M$.
\end{prop}

\begin{proof}
Suppose that $v_1,\ldots,v_r$ jointly generate a proper subrepresentation $0\neq W\subsetneq M$.
We obtain a subrepresentation $\widetilde N = (W,v_1,\dots,v_r)\subset \widetilde M$ of dimension vector $(1,d)$ with $0 < d = \dim W < n$.
We claim that $\widetilde N$ destabilises $\widetilde M$. 
Indeed, the inequality
\begin{equation*}
	\mu_{\theta}(\widetilde N) = \dfrac{\theta_1 + d\theta_2}{1+d} > \dfrac{\theta_1 + n\theta_2}{1+n} = \mu_{\theta}(\widetilde M)
\end{equation*}
holds if and only if $\theta_1 > \theta_2$, which holds by assumption.
Since stability and semistability coincide for this choice of dimension vector, we conclude that $\widetilde M$ is unstable.

For the converse, suppose that $v_1,\ldots,v_r$ jointly generate $M$.	If $\widetilde M$ is not stable, there exists a non-trivial proper destabilising subrepresentation $0 \neq \widetilde N \subsetneq \widetilde M$ of dimension vector $(d_1,d_2)$ with $0 \leq d_1 \leq 1$ and $0 \leq d_2 \leq n$. There are two cases to consider.
	\begin{enumerate}
		\item If $d_1 = 1$, it follows that $d_2 = n$ since $v_1,\ldots,v_r$ jointly generate $M$.
			But then $\widetilde N = \widetilde M$, which is a contradiction.
		\item If $d_1 = 0$ then $d_2 > 0$, and we directly compute
			\begin{equation*}
				\mu_{\theta}(\widetilde N) = \dfrac{d_2\theta_2}{d_2} = \theta_2
					= \dfrac{(1+n)\theta_2}{1+n} < \dfrac{\theta_1+n\theta_2}{1+n} = \mu_{\theta}(\widetilde M)
			\end{equation*}
			because $\theta_2 < \theta_1$.
			But this contradicts the fact that $\widetilde N$ destabilises $\widetilde M$.
	\end{enumerate}
	It follows that $\widetilde M$ is $\theta$-stable.
	This completes the proof.
\end{proof}

\subsection{The non-commutative Quot scheme}

In this section, we write
\[
R=\C\langle x,y,z\rangle
\]
for the free (non-commutative) $\C$-algebra on three generators, and for a complex scheme $B$, we denote by $R_B$ the sheaf of $\O_B$-algebras associated to the presheaf $R\otimes_\C\O_B=\mathscr O_B\braket{x,y,z}$.
We consider the functor 
\begin{equation}\label{eq:NonCommutative_Quot_Functor}
\mathfrak Q_{n,r}\colon\Sch_{\C}^{\op}\ra \Sets
\end{equation}
sending a $\C$-scheme $B$ to the set of isomorphism classes of triples $(M,p,\beta)$, where
\benum
\item $M$ is a left $R_B$-module, locally free of rank $n$ over $\O_B$,
\item $p\colon R_B^r\surj M$ is an $R_B$-linear epimorphism, and
\item $\beta\subset \Gamma(B,M)$ is a basis of $M$ as an $\O_B$-module.
\eenum
Two triples $(M,p,\beta)$ and $(M',p',\beta')$ are considered isomorphic if there is a commutative diagram
\be\label{diag:oblin}
\begin{tikzcd}
R_B^r\arrow[two heads]{r}{p} \arrow[equal]{d} & M\arrow{d}{\varphi} \\
R_B^r\arrow[two heads]{r}{p'} & M'
\end{tikzcd}
\ee
with $\varphi$ an $\O_B$-linear isomorphism transforming $\beta$ into $\beta'$. We denote by $\braket{M,p,\beta}$ the corresponding isomorphism class.

One can also define a functor 
\[
\overline{\mathfrak Q}_{n,r}\colon\Sch_{\C}^{\op}\ra \Sets
\]
by letting $\overline{\mathfrak Q}_{n,r}(B)$ be the set of isomorphism classes of pairs $(M,p)$ just as above, but where no basis of $M$ is chosen. Here, as before, we declare that $(M,p)\sim (M',p')$ when there is a commutative diagram as in \eqref{diag:oblin}. 

Notice that, by considering the kernel of the surjection, an equivalence class $\braket{M,p}$ uniquely determines a left $R_B$-module $N\subset R_B^r$ (such that the quotient $R_B^r/N$ is a locally free $\O_B$-module). The proof of the following result is along the same lines of \cite[Sec.~2]{NCHilb1}.  

\begin{theorem}\label{ncquot}
The smooth quasi-affine scheme $U_{n,r}$ defined in \eqref{unr} represents the functor $\mathfrak Q_{n,r}$, whereas the GIT quotient $U_{n,r}/\GL_n$ represents the functor $\overline{\mathfrak Q}_{n,r}$.
\end{theorem}

\begin{proof}
Let $V=\C^n$ with its standard basis $e_1,\dots,e_n$. Consider the free module $M_0=V\otimes_\C\O_{\mathcal R_{n,r}}$ with basis
$\beta_0=\set{e_j\otimes 1:1\leq j\leq n}$. Let $(X_{ij},Y_{ij},Z_{ij},u_{1}^k,\ldots,u_{r}^k)$ be the coordinates on the affine space $\mathcal R_{n,r}$. Here $1\leq i,j\leq n$ correspond to matrix entries and $1\leq k\leq n$ to vector components.
Then $M_0$ has distinguished elements
\[
\mathsf v_\ell=\sum_{j} e_j\otimes u_{\ell}^j,\qquad 1\leq \ell \leq r. 
\]
For each $\ell = 1,\ldots,r$, consider the $R$-linear map $\theta_\ell\colon R_{\mathcal R_{n,r}} \ra M_0$ given by $\theta_\ell(\mathbb 1)=\mathsf v_\ell$. Then we can construct the map $\theta_0 = \oplus_\ell \theta_\ell\colon R_{\mathcal R_{n,r}}^r\ra M_0$. 
Restricting the triple $(M_0,\theta_0,\beta_0)$ to $U_{n,r}\subset \mathcal R_{n,r}$ gives a morphism of functors 
\[
U_{n,r}\ra \mathfrak Q_{n,r},
\]
whose inverse is constructed as follows.
Let $B$ be a scheme, set again $V=\C^n$ and fix a $B$-valued point $\braket{M,\theta,\beta}\in \mathfrak Q_{n,r}(B)$. The $R$-action on $\beta\subset \Gamma(B,M)$ determines three endomorphisms $(X,Y,Z)\colon B\ra \End(V)^3$. On the other hand, decomposing the map $\theta\colon R_B^r\surj M$ into $r$ maps $\theta_\ell\colon R_B\ra M$ and taking $v_\ell = \theta_\ell(\mathbb 1)$ determines a morphism $(v_1,\ldots,v_r)\colon B\ra V^r$. We have thus constructed a morphism $f\colon B\ra \mathcal R_{n,r}$. The surjectivity of $\theta$ says that $f$ factors through $U_{n,r}$. Therefore $U_{n,r}$ represents $\mathfrak Q_{n,r}$.

Next, let $\pi\colon U_{n,r}\ra \overline U_{n,r} = U_{n,r}/\GL_n$ be the quotient map, which we know is a principal $\GL_n$-bundle.
This implies that $\pi^\ast\colon \QCoh(\overline U_{n,r})\,\widetilde{\ra}\,\QCoh_{\GL_n}(U_{n,r})$ is an equivalence of categories, preserving locally free sheaves~\cite[Prop.~$4.5$]{NCHilb1}. Consider the universal triple $\braket{M_0,\theta_0,\beta_0}$ defined above. Since $M_0$ is a $\GL_n$-equivariant vector bundle on $U_{n,r}$, it follows that, up to isomorphism, there is a unique locally free sheaf $\mathscr M$ on $\overline U_{n,r}$ such that $\pi^\ast \mathscr M\isom M_0$. In fact, $\mathscr M\isom (\pi_\ast M_0)^{\GL_n}\subset \pi_\ast M_0$, the subsheaf of $\GL_n$-invariant sections. The $r$ sections $\mathsf v_\ell$, being $\GL_n$-invariant, descend to sections of $\mathscr M$, still denoted $\mathsf v_\ell$. These generate $\mathscr M$ as an $R_{\overline U_{n,r}}$-module, so we get a surjection $\vartheta\colon R_{\overline U_{n,r}}^r\surj \mathscr M$ sending $\mathbb 1 \mapsto \mathsf v_\ell$ in the $\ell$-th component.
In particular, the pair $\braket{\mathscr M,\vartheta}$ defines a morphism of functors 
\[
\overline U_{n,r}\ra \overline{\mathfrak Q}_{n,r}.
\]
We need to construct its inverse. Let $B$ be a scheme and fix a $B$-valued point $\braket{N,\theta}\in \overline{\mathfrak Q}_{n,r}(B)$. Let $(B_i:i\in I)$ be an open cover of $B$ such that $N_i=N|_{B_i}$ is free of rank $n$ over $\O_{B_i}$. Decompose $\theta = \theta_1\oplus \cdots\oplus \theta_r$ into $r$ maps $\theta_\ell\colon R_B\ra N$. Choose a basis $\beta_i\subset \Gamma(B_i,N_i)$ and let $v_{\ell,i}=\theta_\ell(\mathbb 1)|_{B_i}$ be the restriction of $\theta_\ell(\mathbb 1)\in \Gamma(B,N)$ to $B_i\subset B$. 
As usual, the tuple $(v_{\ell,i}:1\leq \ell\leq r)$ defines a linear surjection $\theta_i\colon R_{B_i}^r\surj N_i$. Each triple $\braket{N_i,\theta_i,\beta_i}$ then defines a point $\psi_i\colon B_i\ra U_{n,r}$, and for all indices $i$ and $j$ there is a matrix $g\in \GL_n(\O_{B_{ij}})$ sending $\beta_i$ to $\beta_j$. In other words, $g$ defines a map $g\colon B_{ij}\ra \GL_n$ such that $g\cdot\psi_i=\psi_j$. Then $\pi\circ \psi_i$ and $\pi\circ \psi_j$ agree on $B_{ij}$, and this determines a unique map to the quotient $f\colon B\ra \overline U_{n,r}$, satisfying $(N,\theta)\sim f^\ast(\mathscr M,\vartheta)$. This shows that $\overline U_{n,r}$ represents $\overline{\mathfrak Q}_{n,r}$.
\end{proof}

As a consequence of this result, the $B$-valued points of the quotient $U_{n,r}/\GL_n$ can be identified with left $R_B$-submodules
\[
N\subset R_B^r
\]
with the property that the quotient is locally free of rank $n$ over $\O_B$. Because it represents the functor of quotients $R^r\surj V_n$, and $R$ is a non-commutative $\C$-algebra, we refer to $U_{n,r}/\GL_n$ as a \emph{non-commutative Quot scheme}, and we introduce the notation
\begin{equation*}
\Quot^n_r = U_{n,r}/\GL_n.
\end{equation*}
Note that $\Quot^n_1$ is the \emph{non-commutative Hilbert scheme}.
Finally, the GIT construction implies that $\Quot^n_r$ is a smooth quasi-projective variety of dimension $2n^2+rn$.

\subsection{The potential}
On the three-loop quiver, consider the potential 
\[
W = x[y,z]\in \C\mathsf L_3,
\]
where $[-,-]$ denotes the commutator.
The associated trace map $U_{n,r}\ra \A^1$, defined by $(A,B,C,v_1,\ldots,v_r)\mapsto \Tr A[B,C]$, is $\GL_n$-invariant.\footnote{The vectors $v_i$ are not involved in the definition of the map. They are only needed to define its domain.}
Thus, it descends to a regular function on the quotient,
\begin{equation}
f_n\colon \Quot^n_r \ra \A^1.
\end{equation}

We now show that $\Quot_{\A^3}(\O^r,n)$ is the scheme-theoretic critical locus of $f_n$.

\begin{theorem}\label{thm:criticallocus}
There is a closed immersion
\[
\Quot_{\A^3}(\O^r,n)\hookrightarrow \Quot^n_r
\]
cut out scheme-theoretically by the exact one-form $\dd f_n$.
\end{theorem}

\begin{proof}
Let $B$ be a scheme. Observe that there is an inclusion of sets
\[
\Quot_{\A^3}(\O^r,n)(B)\subset \Quot^n_r(B).
\]
A $B$-valued point $[N]$ of the non-commutative Quot scheme defines a $B$-valued point of the commutative Quot scheme if and only if the $R$-action on the corresponding $R_B$-module $N$ descends to a $\C[x,y,z]$-action. This happens precisely when the actions of $x$, $y$ and $z$ on $N$ commute with each other. 
Let then $Z\subset \Quot^n_r$ be the image of the zero locus 
\[
\Set{(A,B,C,v_1,\ldots,v_r) \mid [A,B]=[A,C]=[B,C]=0}\subset U_{n,r}
\]
under the quotient map.
Then $[N]$ belongs to $\Quot_{\A^3}(\O^r,n)(B)$ if and only if the corresponding morphism $B\ra \Quot^n_r$ factors through $Z$. But $Z$ agrees, as a scheme, with the critical locus of $f_n$, by \cite[Prop.~$3.8$]{Seg}.
\end{proof}

It follows that $\Quot_{\A^3}(\O^r,n)$ has a symmetric obstruction theory determined by the Hessian of $f_n$. We refer to \cite{BFHilb} for more details on symmetric obstruction theories.

Every scheme $Z$ which can be written as $\set{\dd f = 0}$ for $f$ a regular function on a smooth scheme $U$ has a \emph{virtual motive}, depending on the pair $(U,f)$. According to \cite{BBS}, this is a class 
\[
[Z]_{\vir} \in \mathcal M=K^{\hat\mu}(\Var/\C)\left[\L^{-1/2}\right]
\]
in the $\hat\mu$-equivariant ring of motivic weights, satisfying the property $\chi [Z]_{\vir} = \widetilde\chi(Z)$, where $\widetilde\chi(-)$ denotes the virtual Euler characteristic, as in \eqref{virtualchi}. Here $\L = [\A^1]$ is the Lefschetz motive, and the Euler characteristic homomorphism $\chi\colon K(\Var/\C)\ra \Z$ from the classical Grothendieck ring of varieties is extended to the ring $\mathcal M$ by sending $\L^{1/2}$ to $-1$.
The class 
\[
[Z]_{\vir} = -\L^{-(\dim U)/2}\cdot [\phi_f] \in \mathcal M
\]
is constructed by means of the motivic vanishing cycle class $[\phi_f]$ introduced by Denef--Loeser \cite{DenefLoeser1}. We refer to \cite{BBS} for more details.

Theorem \ref{thm:criticallocus} then produces a virtual motive $[\Quot_{\A^3}(\O^r,n)]_{\vir}\in\mathcal M$ in the ring of equivariant motivic weights, by means of the pair $(\Quot^n_r,f_n)$. Form the generating function
\[
\mathsf Z_r(t) = \sum_{n\geq 0}\, \left[\Quot_{\A^3}(\O^r,n)\right]_{\vir}\cdot t^n\in \mathcal M\llbracket t\rrbracket.
\]
Following the calculation carried out in \cite{BBS} for $r=1$, one can prove the following.

\begin{prop}[{\cite[Prop.~2.3.6]{ThesisR}}]
One has the relation
\be\label{id:mot}
\mathsf Z_{r}(t) = \prod_{m=1}^\infty \prod_{k = 0}^{rm-1}\left(1-\L^{2+k-rm/2}t^m\right)^{-1}.
\ee
\end{prop}

The same motivic formula was proved independently by A.~Cazzaniga \cite[Thm.~2.2.4]{Cazzaniga_Thesis}.

\begin{corollary}
One has the relation
\[
\sum_{n\geq 0}\widetilde\chi(\Quot_{\A^3}(\O^r,n)) t^n = \mathsf M((-1)^rt)^r.
\]
\end{corollary}

\begin{proof}
This follows by applying $\chi$ to \eqref{id:mot}, and using that $\chi(\L^{-1/2}) = -1$.
\end{proof}

From the corollary, along with the observation that 
\[
\sum_{n\geq 0}\chi(\Quot_{\mathbb A^3}(\mathscr O^{\oplus r},n))q^n = \mathsf M(q)^r,
\]
we obtain an identity
\begin{equation}\label{chi_local_quot}
\widetilde\chi(\Quot_{\mathbb A^3}(\mathscr O^{\oplus r},n)) = (-1)^{rn} \chi(\Quot_{\mathbb A^3}(\mathscr O^{\oplus r},n)).
\end{equation}

The Quot scheme $\Quot_{\mathbb A^3}(\mathscr O^{\oplus r},n)$ carries an action by the torus
\[
\mathbf T = (\mathbb C^\times)^3 \times (\mathbb C^\times)^r,
\]
where the $(\mathbb C^\times)^3$-action is the natural lift of the action on $\mathbb A^3$, and thus moves the support of the quotients, and 
the torus $(\mathbb C^\times)^r$ acts by scaling the fibres of $\mathscr O^{\oplus r}$. 

\begin{remark}\label{remark:fixed_loci}
The $(\mathbb C^\times)^3$-fixed locus of $\Quot_{\mathbb A^3}(\mathscr O^{\oplus r},n)$ is compact, 
because it is a closed subscheme of the punctual Quot scheme
\[
\mathsf P_n = \Quot_{\mathbb A^3}(\mathscr O^{\oplus r},n)_0,
\]
the locus of quotients entirely supported at the origin.
But $\mathsf P_n$ is proper, being a fibre of the Quot-to-Chow morphism (which is a proper morphism). 
It follows that the $\mathbf T$-fixed locus
\[
\Quot_{\mathbb A^3}(\mathscr O^{\oplus r},n)^{\mathbf T} \subset \Quot_{\mathbb A^3}(\mathscr O^{\oplus r},n)^{(\mathbb C^\times)^3}
\]
is also compact. We give a precise description in Lemma \ref{lemma:T_fixed_points} below.
\end{remark}

For later purpose, we now determine the $\mathbf T$-fixed locus for this action.

\begin{lemma}\label{lemma:T_fixed_points}
There is an isomorphism of schemes
\[
\Quot_{\mathbb A^3}(\mathscr O^{\oplus r},n)^{\mathbf T} = \coprod_{n_1+\cdots+n_r=n}\prod_{i=1}^r \Hilb^{n_i}(\mathbb A^3)^{(\mathbb C^\times)^3}.
\]
In particular, the $\mathbf T$-fixed locus is isolated and compact. It parameterises direct sums $\mathscr I_{Z_1} \oplus \cdots \oplus \mathscr I_{Z_r}$ of monomial ideals.
\end{lemma}

\begin{proof}
The main result of \cite{Bifet} shows that 
\begin{equation}\label{Bifet_theorem}
\Quot_{\mathbb A^3}(\mathscr O^{\oplus r},n)^{(\mathbb C^\times)^r} = \coprod_{n_1+\cdots+n_r=n}\prod_{i=1}^r \Hilb^{n_i}(\mathbb A^3).
\end{equation}
The claimed isomorphism follows by taking $(\mathbb C^\times)^3$-invariants. Since $\Hilb^{k}(\mathbb A^3)^{(\mathbb C^\times)^3}$ is isolated (a disjoint union of reduced points, each corresponding to a plane partition of $k$), the result follows.
\end{proof}

\section{Quotients of a locally free sheaf on an arbitrary 3-fold}\label{sec:thmA}

The goal of this section is to prove Theorem \ref{Thm:LocallyFree}. We follow the cut-and-paste technique of Behrend--Fantechi \cite{BFHilb}, also used in \cite{LocalDT}.

Let $X$ be a smooth quasi-projective $3$-fold and let $\F$ be a locally free sheaf of rank $r$. We will show that
\be\label{signedchi}
\chi(\Quot_X(\F,n),\nu) = (-1)^{rn}\chi(\Quot_X(\F,n)).
\ee
This was proved for $\F = \O_X$ in \cite{BFHilb,LEPA,JLI} and for the torsion free sheaf $\F=\mathscr I_C$, where $C\subset X$ is a smooth curve, in \cite{LocalDT}.
Combined with the Euler characteristic calculation \cite[Thm.~1.1]{Gholampour2017} recalled in \eqref{GKformula}, formula \eqref{signedchi} proves Theorem \ref{Thm:LocallyFree}.

\subsection{The virtual Euler characteristic}\label{behrendstuff}
Let $X$ be a complex scheme, and let $\nu_X\colon X(\C)\ra \Z$ be its Behrend function \cite{Beh}. The virtual (or weighted) Euler characteristic of $X$, as recalled in \eqref{virtualchi}, is the integer
\begin{equation*}
\widetilde\chi(X) \defeq \chi(X,\nu_X) \defeq \sum_{k\in \Z}k\cdot \chi(\nu_X^{-1}(k)).
\end{equation*}
Given a morphism $f\colon Z\ra X$, one defines the \emph{relative} virtual Euler characteristic by $\widetilde\chi(Z,X) = \chi(Z,f^\ast \nu_X)$.
We now recall the properties of $\nu$ and $\widetilde\chi$ we will need later. 
\begin{enumerate}
\item [$\circ$] If $f\colon Z\ra X$ is \'etale, then $\nu_Z = f^\ast \nu_X$.
\item [$\circ$] If $Z_1$, $Z_2\subset X$ are disjoint locally closed subschemes, one has
\be\label{ChiSum}
\widetilde\chi(Z_1\amalg Z_2,X) = \widetilde\chi(Z_1,X)+\widetilde\chi(Z_2,X).
\ee
\item [$\circ$] Given two morphisms $Z_i \ra X_i$, one has
\be\label{ChiProduct}
\widetilde\chi(Z_1\times Z_2,X_1\times X_2)=\widetilde\chi(Z_1,X_1)\cdot \widetilde\chi(Z_2,X_2).
\ee
\item [$\circ$] If one has a commutative diagram
\[
\begin{tikzcd} 
Z \arrow{r}{} \arrow{d}{}
&X \arrow{d}{}\\
W \arrow{r}{} &Y
\end{tikzcd}
\]
with $X\ra Y$ smooth and $Z\ra W$ finite \'etale of degree $d$, then
\be\label{ChiSquare}
\widetilde\chi(Z,X)=d(-1)^{\dim X/Y}\widetilde\chi(W,Y).
\ee
As a special case, if $X\ra Y$ is \'etale (for instance, an open immersion) and $Z\ra X$ is a morphism, then 
\be\label{ChiEtale}
\widetilde \chi(Z,X)=\widetilde\chi(Z,Y).
\ee
\end{enumerate}

\subsection{Reduction to the deepest stratum}
Consider the Quot-to-Chow morphism
\[
\sigma\colon\Quot_X(\F,n) \ra \Sym^nX
\]
sending $[\F\onto Q]$ to the zero cycle given by the support of $Q$. This is a morphism of schemes by \cite[Cor.~7.15]{Rydh1}. For every partition $\alpha$ of $n$, we have the locally closed subscheme $\Sym^n_\alpha X$ of $\Sym^nX$ parametrising zero-cycles with multiplicities distributed according to $\alpha$. The restrictions
\[
\sigma_\alpha\colon \Quot_X(\F,n)_\alpha \ra \Sym^n_\alpha X
\]
induce a locally closed stratification of the Quot scheme.
The most interesting stratum is the deepest one, corresponding to the full partition $\alpha = (n)$.
The map 
\be\label{mor:phc}
\sigma_{(n)}\colon\Quot_X(\F,n)_{(n)}\ra X
\ee
has fibre over a point $p\in X$ the \emph{punctual} Quot scheme $\Quot_X(\F,n)_p$, parametrising quotients whose target is supported entirely at $p$. Note that
\[
\Quot_X(\F,n)_p\isom\Quot_X(\O_X^r,n)_p\isom \Quot_{\A^3}(\O^r,n)_0,
\]
where for the first isomorphism we need $\F$ locally free.
We will see in Lemma \ref{lemma:ZLT} below that \eqref{mor:phc} is Zariski locally trivial with fibre the punctual Quot scheme. From now on we shorten
\[
\mathsf P_n = \Quot_{\A^3}(\O^r,n)_0\,,\quad \nu_n = \nu_{\Quot_{\A^3}(\O^r,n)}\big|_{\mathsf P_n}.
\]
Given a partition $\alpha = (\alpha_1,\ldots,\alpha_{s_\alpha})$ of $n$, we set 
\[
Q_\alpha = \prod_{i = 1}^{s_\alpha}\Quot_X(\F,\alpha_i),
\]
and we let 
\[
V_\alpha \subset Q_\alpha
\]
be the open subscheme parametrising quotient sheaves with pairwise disjoint support. Then, by the results of Section \ref{quotmess} (cf.~Proposition \ref{prop:etaleV}), $V_\alpha$ admits an \'etale  map $f_\alpha$ to the Quot scheme $\Quot_X(\F,n)$, and we let $U_\alpha$ denote its image. 
Note that $U_\alpha$ contains the stratum $\Quot_X(\F,n)_\alpha$ as a closed subscheme. We can then form the cartesian square
\be\label{diag1313}
\begin{tikzcd}[row sep=large,column sep=normal]  
Z_\alpha\MySymb{dr}\arrow[hook]{r}{} \arrow[swap]{d}{\textrm{Galois}} & V_\alpha \arrow{d}{f_\alpha}\arrow[hook]{r}{\textrm{open}} & Q_\alpha\\
\Quot_X(\F,n)_\alpha \arrow[hook]{r}{} & U_\alpha\arrow[hook]{r}{\textrm{open}} & \Quot_X(\F,n)
\end{tikzcd}
\ee
defining $Z_\alpha$. The leftmost vertical map is finite \'etale with Galois group $G_\alpha$, the automorphism group of the partition. We observe that $Z_\alpha = \prod_i\Quot_X(\F,\alpha_i)_{(\alpha_i)}\setminus \widetilde\Delta$, where $\widetilde\Delta$ parametrises $s_\alpha$-tuples of sheaves with intersecting supports. Therefore we have a second fibre square
\be\label{diag6565}
\begin{tikzcd}[row sep=large,column sep=normal]  
Z_\alpha\MySymb{dr} \arrow[hook]{r}{} \arrow[swap]{d}{\pi_\alpha} & 
\prod_i\Quot_X(\F,\alpha_i)_{(\alpha_i)}\arrow{d} \\
X^{s_\alpha}\setminus \Delta \arrow[hook]{r} & X^{s_{\alpha}}
\end{tikzcd}
\ee
where the vertical map on the right is the product of punctual Quot-to-Chow morphisms \eqref{mor:phc}, the horizontal inclusions are open immersions, and $\Delta \subset X^{s_{\alpha}}$ denotes the big diagonal.

\subsection{Calculation}
Recall that the absolute $\widetilde\chi$ is not additive on strata, but that the $\widetilde\chi$ relative to a morphism is additive.
Exploiting the diagram \eqref{diag1313}, we compute

\begin{align*}
\widetilde\chi(\Quot_X(\F,n)) &=
\sum_{\alpha}\widetilde\chi\left(\Quot_X(\F,n)_\alpha,\Quot_X(\F,n)\right) 
& \textrm{by }\eqref{ChiSum}\\
&= \sum_{\alpha}\widetilde\chi\left(\Quot_X(\F,n)_\alpha,U_\alpha\right) 
& \textrm{by }\eqref{ChiEtale}\\
&= \sum_{\alpha}|G_\alpha|^{-1}\widetilde\chi(Z_\alpha,V_\alpha)
& \textrm{by }\eqref{ChiSquare}\\
&= \sum_{\alpha}|G_\alpha|^{-1}\widetilde\chi\left(Z_\alpha,Q_\alpha\right).
& \textrm{by }\eqref{ChiEtale}\\
\end{align*}

Before giving an expression of $\widetilde\chi(Z_\alpha,Q_\alpha)$ via the fibre square \eqref{diag6565}, we make a few technical observations.

\begin{lemma}\label{lemma:punctual}
We have a canonical isomorphism
\[
\A^3\times \mathsf P_n\,\widetilde{\ra}\,\Quot_{\A^3}(\O^r,n)_{(n)}
\]
and the restriction of the Behrend function of $\Quot_{\A^3}(\O^r,n)$ to its deepest stratum is the pullback of $\nu_n$ along the projection to $\mathsf P_n$.
\end{lemma}

\begin{proof}
This follows by standard arguments, see \cite[Lemma 4.6]{BFHilb} or \cite[Prop. 3.1]{LocalDT}.
\end{proof}

\begin{lemma}\label{lemma:chiPn}
One has $\widetilde\chi(\Quot_{\A^3}(\O^r,n)) = (-1)^{rn}\chi(\mathsf P_n) = \chi(\mathsf P_n,\nu_{n})$.
\end{lemma}

\begin{proof}
The first identity follows by Equation \eqref{chi_local_quot}, along with the observation that 
\[
\chi(\Quot_{\A^3}(\O^r,n)) = \chi(\mathsf P_n),
\]
which holds because $\mathsf P_n$ is a $\mathbf T$-invariant subscheme containing the $\mathbf T$-fixed locus (cf.~Remark \ref{remark:fixed_loci}).
By \cite[Cor.~3.5]{BFHilb}, we deduce from the first identity that the parity of the tangent space dimension at a $\mathbf T$-fixed point of the Quot scheme is $(-1)^{rn}$. This can also be proved directly by virtual localisation.
The $\mathbf T$-action on $\Quot_{\A^3}(\O^r,n)$ has isolated fixed points by Lemma \ref{lemma:T_fixed_points}, and using again that $\mathsf P_n$ is $\mathbf T$-invariant and contains the $\mathbf T$-fixed locus, we obtain the second identity by another application of \cite[Cor.~3.5]{BFHilb}.
\end{proof}

\begin{lemma}\label{lemma:ZLT}
Let $X$ be a smooth $3$-fold, $\F$ a locally free sheaf of rank $r$, and denote by $\nu_Q$ the Behrend function of $\Quot_X(\F,n)$. Let $\sigma_{(n)}$ be the morphism defined in \eqref{mor:phc}. Then there is a Zariski open cover $(U_i)$ of $X$ such that 
\[
\left(\sigma_{(n)}^{-1}U_i,\nu_{Q}\right) \isom (U_i,\mathbb{1}_{U_i})\times \left(\mathsf P_n,\nu_n\right)
\]
as schemes with constructible functions on them. 
\end{lemma}

\begin{proof}
This is a standard argument. See for instance \cite[Cor.~3.2]{LocalDT}.
\end{proof}

\begin{lemma}\label{lemma:Qalpha}
One has the identity
\[
\widetilde\chi(Z_\alpha,Q_\alpha)
=\chi(X^{s_{\alpha}}\setminus \Delta)\cdot \prod_{i=1}^{s_\alpha}\chi(\mathsf P_{\alpha_i},\nu_{\alpha_i}).
\]
\end{lemma}

\begin{proof}
By Lemma \ref{lemma:ZLT}, we can find a Zariski open cover $X^{s_\alpha}\setminus \Delta = \bigcup_jB_j$ such that 
\be\label{zltFibr}
\left(\pi_\alpha^{-1}B_j,\nu_{Q_\alpha}\right) = 
(B_j,\mathbb 1_{B_j}) \times \left(\prod_i \mathsf P_{\alpha_i},\prod_i \nu_{\alpha_i}\right),
\ee
where $\pi_\alpha\colon Z_\alpha\ra X^{s_\alpha}\setminus \Delta$ is the map that appeared for the first time in \eqref{diag6565}.
Such an open covering can be turned into a locally closed stratification $X^{s_\alpha}\setminus \Delta = \amalg_{\ell} U_\ell$ such that each $U_\ell$ is contained in some $B_j$. Then we have
\begin{align*}
\widetilde\chi(Z_\alpha,Q_\alpha)
&= \sum_\ell \widetilde\chi\left(\pi_\alpha^{-1}U_\ell,Q_\alpha\right) & \textrm{by }\eqref{ChiSum}\\
&= \sum_\ell\chi\bigl(U_\ell \times \prod_i \mathsf P_{\alpha_i},\mathbb 1_{U_\ell}\times \prod_i \nu_{\alpha_i}\bigr) & \textrm{by } \eqref{zltFibr}\\
&= \sum_\ell\chi(U_\ell,\mathbb 1_{U_\ell}) \cdot \prod_i\chi\left(\mathsf P_{\alpha_i},\nu_{\alpha_i}\right) & \textrm{by }\eqref{ChiProduct}\\
&= \chi(X^{s_\alpha}\setminus \Delta) \cdot \prod_i\chi\left(\mathsf P_{\alpha_i},\nu_{\alpha_i}\right),
\end{align*}
and the lemma is proved.
\end{proof}

We are now in a position to finish the computation of $\widetilde{\chi}(\Quot_X(\F,n))$:
\begin{align*}
\widetilde\chi(\Quot_X(\F,n)) &=
\sum_{\alpha}|G_\alpha|^{-1}\chi(X^{s_{\alpha}}\setminus \Delta)\cdot \prod_{i}\chi\left(\mathsf P_{\alpha_i},\nu_{\alpha_i}\right) & \textrm{by Lemma }\ref{lemma:Qalpha} \\
&=\sum_{\alpha}|G_\alpha|^{-1}\chi(X^{s_{\alpha}}\setminus \Delta)\cdot \prod_{i}\,(-1)^{r\alpha_i}\chi\left(\mathsf P_{\alpha_i}\right) & \textrm{by Lemma }\ref{lemma:chiPn}\\
&=(-1)^{rn}\sum_{\alpha}|G_\alpha|^{-1}\chi(X^{s_{\alpha}}\setminus \Delta)\cdot \prod_{i}\chi\left(\mathsf P_{\alpha_i}\right) \\
&=(-1)^{rn}\chi(\Quot_X(\F,n)).
\end{align*}
By virtue of formula \eqref{GKformula} of Gholampour--Kool \cite[Thm.~1.1]{Gholampour2017}, we conclude that for a locally free sheaf $\F$ on a smooth quasi-projective $3$-fold $X$ one has
\begin{equation}\label{eq:Thm_Locally_Free_Case}
\sum_{n\geq 0}\widetilde\chi(\Quot_X(\F,n)) q^n = \mathsf M((-1)^rq)^{r\chi(X)}.
\end{equation}
This completes the proof of Theorem \ref{Thm:LocallyFree}.

\section{Tools for the proof of Theorem \ref{thm2}}
This section contains the technical preliminaries we will need in Section \ref{Sec:dtpt} for the proof of Theorem \ref{thm2}. We mainly follow \cite{Toda2}.
The key objects are the following.
\bitem
\item [(i)] Moduli spaces of torsion free sheaves and of PT pairs on a smooth projective $3$-fold $X$. These can be seen as parametrising stable objects in a suitable heart 
\[
\mathscr A_\mu\subset D^b(X)
\]
of a bounded t-structure on the derived category of $X$.
\item [(ii)] The Hall algebra of the abelian category $\mathscr A_\mu$ and the associated (Behrend weighted, unweighted) integration maps.
\eitem
Throughout, let $X$ be a smooth projective $3$-fold, not necessarily Calabi--Yau.
We fix an ample class $\omega$ on $X$, an integer $r\geq 1$, and a divisor class $D\in H^2(X,\Z)$, satisfying the coprimality condition
\be\label{coprime}
\gcd(r,D\cdot \omega^2) = 1.
\ee

\subsection{Moduli of sheaves and PT pairs}\label{sec:Moduli}
Slope stability with respect to $\omega$ is defined in terms of the slope function $\mu_\omega$, attaching to a torsion free coherent sheaf $E$ the ratio
\[
\mu_\omega(E) = \frac{c_1(E)\cdot \omega^2}{\rk E}\in \Q\cup\set{\infty}.
\]
The coherent sheaf $E$ is $\mu_{\omega}$-stable or slope-stable if the strict inequality
\[
\mu_{\omega}(S) < \mu_{\omega}(E)
\] 
holds for all proper non-trivial subsheaves $0 \neq S \subsetneq E$ with $0 < \rk(S) < \rk(E)$.
The sheaf is slope-semistable if the same condition holds with $<$ replaced by $\leq$.
Note that condition \eqref{coprime} implies that any slope-semistable sheaf is slope-stable and, hence, that the notions of slope and Gieseker stability coincide; see~\cite{modulisheaves}.
In particular, the coarse moduli space
\be\label{MDT}
M_{\DT}(r,D,-\beta,-m)
\ee
of $\mu_\omega$-stable sheaves of Chern character $(r,D,-\beta,-m)$ is a projective scheme \cite{modulisheaves}.

If $X$ is Calabi--Yau, $M_{\DT}(r,D,-\beta,-m)$ carries a symmetric perfect obstruction theory and hence a (zero-dimensional) virtual fundamental class by \cite{ThomasThesis}.
By the main result of \cite{Beh}, the associated DT invariant
\be\label{DThigherrank}
\DDT(r,D,-\beta,-m) = \int_{\left[M_{\DT}(r,D,-\beta,-m)\right]^{\vir}}1 \in \Z,
\ee
coincides with Behrend's virtual Euler characteristic $\widetilde\chi(M_{\DT}(r,D,-\beta,-m))$.

The notion of higher rank PT pair originated in the work of Lo \cite{Lo}, and was revisited by Toda \cite{Toda2}; also see the example \cite[Ex.~3.4]{Toda2}.
The definition is the following.
\begin{definition}[{\cite[Def.~3.1]{Toda2}}]\label{ptpair}
	A \emph{PT pair} on $X$ is a two-term complex $J^\bullet \in D^b(X)$ such that
    \begin{enumerate}
    	\item $H^i(J^\bullet) = 0$ if $i \neq 0,1$,
    	\item $H^0(J^\bullet)$ is $\mu_{\omega}$-(semi)stable and $H^1(J^\bullet)$ is zero-dimensional,
        \item $\Hom(Q[-1],J^\bullet) = 0$ for every zero-dimensional sheaf $Q$.
    \end{enumerate}
We say that a PT pair $J^\bullet$ is $\F$-\emph{local}, or \emph{based at} $\F$, if $H^0(J^\bullet)=\F$.
\end{definition}
By \cite[Thm.~1.2]{Lo}, the coarse moduli space
\be\label{MPT}
M_{\PT}(r,D,-\beta,-m)
\ee
parametrising PT pairs with the indicated Chern character is a proper algebraic space of finite type.
If $X$ is in addition Calabi--Yau, the PT moduli space carries a symmetric perfect obstruction theory by the results of Huybrechts--Thomas \cite{HT}.
The PT invariant is defined, just as \eqref{DThigherrank}, by integration against the associated virtual class.
Similarly to the DT case, by \cite{Beh} it coincides with Behrend's virtual Euler characteristic
\[
\PPT(r,D,-\beta,-m) = \widetilde\chi(M_{\PT}(r,D,-\beta,-m)).
\]

For later purposes, we recall the notion of \emph{family} of PT pairs.
\begin{definition}\label{def:familyPT}
A \emph{family of PT pairs} parametrised by a scheme $B$ is a perfect complex $J^\bullet \in D^b(X\times B)$ such that, for every $b\in B$,
the derived restriction of $J^\bullet$ to $X\times \set{b}\subset X\times B$ is a PT pair.
\end{definition}

Note that no additional flatness requirement is imposed; see Remark~\ref{flatness} for a brief discussion on \emph{families of objects} in a heart of a bounded t-structure on $D^b(X)$.

We define the moduli spaces
\be\label{moduli}
M_{\DT}(r,D),\quad M_{\PT}(r,D)
\ee
of $\mu_\omega$-stable sheaves and PT pairs, respectively, as the (disjoint) unions of the moduli spaces \eqref{MDT} and \eqref{MPT}, over all $(\beta,m)\in H^4(X)\oplus H^6(X)$.
Elements of these moduli spaces will be called DT and PT objects respectively.

\begin{remark}\label{rem:stacks_gerbes}
The moduli \emph{stacks} of DT and PT objects 
\be\label{DTPT_stacks}
\mathcal M_{\DT}(r,D,-\beta,-m),\quad \mathcal M_{\PT}(r,D,-\beta,-m)
\ee
also enter the picture in Toda's proof of the DT/PT correspondence in the key Hall algebra identity \cite[Lemma 3.16]{Toda2} underlying his wall-crossing formula.
By the coprimality assumption \eqref{coprime}, each sheaf or PT pair parametrised by these stacks has automorphism group $\G_m$. Thus, the stacks $\mathcal M_{\DT}(r,D)$ and $\mathcal M_{\PT}(r,D)$, defined as the disjoint unions of the stacks \eqref{DTPT_stacks} over all $(\beta,m)$, are $\G_m$-gerbes over their coarse moduli spaces \eqref{moduli}.
\end{remark}

\subsubsection{A characterisation of PT pairs}
Let $C\subset X$ be a curve, where $X$ is a Calabi--Yau $3$-fold. The formulation of the $C$-local DT/PT correspondence requires $C$ to be Cohen--Macaulay. This is equivalent to the ideal sheaf $\mathscr I_C\subset \O_X$ being both a DT and a PT object.
Similarly, an $\mathcal F$-local (higher rank) DT/PT correspondence requires the sheaf $\mathcal F$ to be both a DT object and a PT object. This assumption will be crucial, for instance, in the proof of Lemma \ref{Closed_Points_13}.

We introduce the following definition, which allows us to recognise a PT pair more easily and to describe the intersection of the spaces \eqref{moduli} of DT and PT objects.

\begin{definition}
A complex $J^\bullet \in D^b(X)$ satisfying properties (1) and (2) in Definition \ref{ptpair} is called a \emph{pre-PT pair}.
\end{definition}

Whether or not a pre-PT pair is a PT pair depends on the vanishing of a particular local Ext sheaf.
To see this, recall that the derived dual $(-)^{\vee} = {\textbf R}\lHom(-,\O_X)$ is a triangulated anti-equivalence of $D^b(X)$ that restricts to an anti-equivalence of abelian categories
\[
(-)^{\vee}\colon \Coh_0(X)\,\widetilde{\ra}\,\Coh_0(X)[-3],\qquad Q\mapsto Q^{\textrm{D}}[-3],
\]
where $Q^{\textrm{D}} = \lExt^3(Q,\O_X)$ is the usual dual of a zero-dimensional sheaf on a $3$-fold.
In particular, we also have an anti-equivalence $\Coh_0(X)[-1]\,\widetilde{\ra}\,\Coh_0(X)[-2]$.
\begin{lemma}\label{lem:PrePT_is_PT}
Let $J^\bullet$ be a pre-PT pair.
Then $J^\bullet$ is a PT pair if and only if $\lExt^2(J^\bullet,\O_X)=0$.
\end{lemma}
\begin{proof}
There is a natural truncation triangle
\be\label{triangle}
\tau_{\leq 1}(J^{\bullet \vee})\ra J^{\bullet \vee}\ra
\lExt^2(J^\bullet,\O_X)[-2],
\ee
where $\tau_{\leq 1}(J^{\bullet \vee})$ and $J^{\bullet \vee}$ live in degrees $[0,1]$ and $[0,2]$ respectively.
Note that the rightmost object is a (shifted) zero-dimensional sheaf because $H^0(J^{\bullet})$ is torsion free and $H^1(J^{\bullet})$ is zero-dimensional.
The dualising involution yields
\[
\Hom(Q[-1],J^\bullet) = \Hom(J^{\bullet \vee},Q^{\textrm{D}}[-2])
\]
for every zero-dimensional sheaf $Q \in \Coh_0(X)$.
But we also have
\[
\Hom(J^{\bullet \vee},Q^{\textrm{D}}[-2]) = \Hom(\lExt^2(J^{\bullet},\O_X),Q^{\textrm{D}})
\]
by \eqref{triangle}, since there are no maps from higher to lower degree in $D^b(X)$.
In conclusion, it follows that $J^{\bullet}$ is a PT pair if and only if $\lExt^2(J^{\bullet},\O_X) = 0$ as claimed.
\end{proof}

\begin{example}
In the case of a classical PT pair $I^\bullet=[\O_X\ra F]$, where $F$ is a pure one-dimensional sheaf, the above vanishing can be deduced directly from the existence of the triangle $I^\bullet\ra \O_X\ra F$.
Indeed, it induces an exact sequence
\[
\cdots \to \lExt^2(\O_X,\O_X)\ra \lExt^2(I^\bullet,\O_X)\ra 
\lExt^3(F,\O_X) \to \cdots
\]
in which the outer two terms vanish, the latter by purity of $F$.
\end{example}

The following corollary explains the assumptions on the sheaf  $\mathcal{F}$ made in \cite{Gholampour2017}, and the assumptions required to have an $\F$-local DT/PT correspondence.
\begin{corollary}\label{cor:DTPT_object}
Let $\mathcal F$ be a coherent sheaf on $X$.
Then $\mathcal F$ is both a DT and a PT object if and only if $\mathcal F$ is a $\mu_{\omega}$-stable sheaf of homological dimension at most one.
\end{corollary}
\begin{proof}
The coherent sheaf $\F$ is $\mu_{\omega}$-stable if and only if it is a DT object.
It is then torsion free, so $\lExt^{\geq 3}(\mathcal F,\O_X) = 0$.
In addition, $\F$ has homological dimension at most one if and only if $\lExt^{\geq 2}(\mathcal F,\O_X)=0$, but this holds if and only if it is a PT object by Lemma~\ref{lem:PrePT_is_PT}.
\end{proof}

\begin{remark}\label{rmk:vanishingext}
Corollary \ref{cor:DTPT_object} implies that for every smooth projective $3$-fold $X$, for every $\mu_{\omega}$-stable torsion free sheaf $\mathcal F$ of homological dimension at most one, and for every zero-dimensional sheaf $Q$, the only possibly non-vanishing Ext groups between $\mathcal F$ and $Q$ are
\[
\Hom(\mathcal F,Q)\isom \Ext^3(Q,\mathcal F)^\ast,\qquad \Ext^1(\mathcal F,Q)\isom \Ext^2(Q,\mathcal F)^\ast.
\]
This in fact holds without the stability assumption by \cite[Lemma 4.1]{Gholampour2017}.
As a consequence, one has $\hom(\mathcal F,Q)-\ext^1(\mathcal F,Q)=r\cdot \ell(Q)$.
Note that the last relation reads
\begin{equation}
    \ext^1(\F,Q[-1]) - \ext^1(Q[-1],\F) = r \cdot \ell(Q),
\end{equation}
which is reminiscent of the DT/PT wall-crossing.
\end{remark}

\subsection{Hall algebras and integration maps}
We recall from \cite[Sec.~2.5,~3.4]{Toda2} the Hall algebra of a suitable heart $\catA_\mu$ of a bounded t-structure on $D^b(X)$. 
To introduce it, recall that a \emph{torsion pair} \cite{MR1327209} on an abelian category $\catA$ is a pair of full subcategories $(\catT,\catF)$ such that
\begin{enumerate}
    \item for $T \in \catT$ and $F \in \catF$, we have $\Hom(T,F) = 0$,
    \item for every $E \in \catA$, there exists a short exact sequence
        \[
        0 \to T_E \to E \to F_E \to 0
        \]
        in $\catA$ with $T_E \in \catT$ and $F_E \in \catF$. This sequence is unique by (1).
\end{enumerate}

Recall also that we have fixed an ample class $\omega$ on the $3$-fold $X$ and a pair $(r,D)$ satisfying the coprimality condition \eqref{coprime}.
Define the number
\[
\mu = \frac{D\cdot \omega^2}{r}\in \Q.
\]
Toda \cite{Toda2} considers the torsion pair
\[
\Coh(X) = \langle \Coh_{>\mu}(X),\Coh_{\leq\mu}(X) \rangle,
\]
where for an interval $I\subset \R\cup\set{\infty}$, the category $\Coh_I(X)$ is the extension-closure of the $\mu_\omega$-semistable objects $E\in \Coh(X)$ with slope $\mu_\omega(E)\in I$, along with the zero sheaf.
Tilting $\Coh(X)$ at the above torsion pair produces the heart
\[
\mathscr A_\mu = \langle \Coh_{\leq\mu}(X),\Coh_{>\mu}(X)[-1]\rangle
\]
of a bounded t-structure on $D^b(X)$.
We summarise the required properties of $\catA_{\mu}$.
\begin{lemma}
The category $\catA_{\mu}$ satisfies the following properties.
\begin{enumerate}
    \item An object $E \in \catA_{\mu}$ is a two-term complex in $D^b(X)$ such that
    \[
    H^0(E) \in \Coh_{\leq \mu}(X), \quad H^1(E) \in \Coh_{> \mu}(X), \quad \text{and } H^i(E) = 0 \text{ otherwise}.
    \]
    \item $\catA_{\mu}$ is a noetherian and abelian category,
    \item $\catA_{\mu}$ contains the noetherian and abelian full subcategory
    \[
    \catB_\mu = \langle \Coh_{\mu}(X),\Coh_{\leq 1}(X)[-1]\rangle \subset \catA_\mu,
    \]
    where $\Coh_{\leq 1}(X)$ consists of sheaves with support of dimension at most one,
    \item the category $\Coh_{\leq 1}(X) \subset \catA_{\mu}$ is closed under subobjects, extensions, and quotients.
\end{enumerate}
\end{lemma}

Let $\mathscr M_X$ be Lieblich's moduli stack of semi-Schur objects on $X$; it is an algebraic stack locally of finite type \cite{Lieblich1}. Recall that a complex $E\in D^b(X)$ is a semi-Schur object if $\Ext^i(E,E) = 0$ for all $i < 0$ (see loc.~cit.~for more details).
Because $\mathscr A_\mu$ is the heart of a bounded t-structure, its objects have no negative self-extensions.
Consider the substack of $\mathscr M_X$ parametrising objects in $\catA_{\mu}$.
It defines an open substack of $\mathscr M_X$ (for example by \cite[Prop.~4.11]{MR3905133}) which we abusively\footnote{If $\mathscr C$ is a subcategory of $D^b(X)$ whose objects have no negative self-extensions, we abuse notation and denote the corresponding moduli stack, a substack of $\mathscr M_X$, by $\mathscr C$ if it exists.} still denote by $\catA_{\mu}$. The category $\mathscr B_\mu$ defines, in turn, an open substack of $\mathscr A_\mu$; this follows from the arguments and results in~\cite[App.~A]{MR2998828}, in particular~\cite[Thm.~A.8]{MR2998828}. It follows that both $\mathscr{B}_{\mu}$ and $\mathscr{A}_{\mu}$ are algebraic stacks locally of finite type. The moduli stacks \eqref{DTPT_stacks} both carry an open immersion to $\mathscr B_\mu$, cf.~\cite[Remark 3.8]{Toda2}.

\begin{remark}\label{flatness}
Let $B$ be a base scheme.
Since $X$ is smooth and $\catA_{\mu} \subset \mathscr M_X$ is open, a $B$-valued point of $\catA_{\mu}$ is a perfect complex $E$ on $X \times B$ such that $E_b = \textbf{L} i_b^*(E) \in \catA_{\mu}$ for every closed point $b \in B$. (Here $i_b\colon X\times \set{b} \into X\times B$ is the natural inclusion.)
Note that no further flatness assumption is required.

Indeed, Bridgeland shows in \cite[Lemma~4.3]{Bri2} that a complex $E \in D^b(X \times B)$ is a (shifted) sheaf \emph{flat over $B$} if and only if each \emph{derived} fibre $E_b \in \Coh(X)$ is a sheaf.
In other words, flatness of $E$ over $B$ is encoded in the vanishing $H^i(E_b) = 0$ for all $i < 0$.
Thus the right notion of a $B$-family of objects in $\catA_{\mu}$ is given by a perfect complex $E$ on $X \times B$ such that $H^i_{\catA_{\mu}}(E_b) = 0$ for all $i \neq 0$ and for all closed points $b \in B$. This means precisely that $E_b \in \catA_{\mu}$.
\end{remark}

\subsubsection{Motivic Hall algebra}
We recall the definition of the motivic Hall algebra from \cite{Bri-Hall}.
Let $S$ be an algebraic stack that is \emph{locally} of finite type with affine geometric stabilisers.
A (representable) morphism of stacks is a \emph{geometric bijection} if it induces an equivalence on $\C$-valued points. It is a \emph{Zariski fibration} if its pullback to any scheme is a Zariski fibration of schemes.
\begin{definition}\label{def:Relative_Grothendieck_Group_of_Stacks}
The $S$-relative Grothendieck group of stacks is the $\Q$-vector space $K(\St/S)$ generated by symbols $[T \to S]$, where $T$ is a \emph{finite type} algebraic stack over $\C$ with affine geometric stabilisers, modulo the following relations.
	\begin{enumerate}
		\item For every pair of $S$-stacks $f_1, f_2 \in \St/S$, we have 
			\begin{equation*}
				[T_1 \sqcup T_2 \xrightarrow{f_1 \sqcup f_2} S] =
					[T_1 \xrightarrow{f_1} S] + [T_2 \xrightarrow{f_2} S].
			\end{equation*}
		\item For every geometric bijection $T_1 \to T_2$ of $S$-stacks, we have
			\begin{equation*}
				[T_1 \xrightarrow{f_1} S] = [T_2 \xrightarrow{f_2} S].
			\end{equation*}
		\item For every pair of Zariski locally trivial fibrations $f_i \colon T_i \to Y$
			with the same fibres and every morphism $g \colon Y \to S$, we have
			\begin{equation*}
				[T_1 \xrightarrow{g \circ f_1} S] = [T_2 \xrightarrow{g \circ f_2} S].
			\end{equation*}
	\end{enumerate}
\end{definition}
\begin{remark}
The last relation plays no further role in this paper. 
\end{remark}
  
As a $\Q$-vector space, the \emph{motivic Hall algebra} of $\catA_{\mu}$ is 
\[
H(\catA_{\mu}) = K(\St/\catA_{\mu}).
\]
We now define the product $\star$ on $H(\mathscr A_\mu)$.
Let $\mathscr A_\mu^{(2)}$ be the stack of short exact sequences in 
$\mathscr A_\mu$, and let 
\[
p_i\colon \mathscr A_\mu^{(2)}\ra \mathscr A_\mu,\qquad i=1,2,3,
\]
be the $1$-morphism sending a short exact sequence $0\ra E_1\ra E_3\ra E_2\ra 0$ to $E_i$.

The following is shown in \cite[App.~B]{BCR} for every heart of a bounded t-structure on $D^b(X)$ whose moduli stack is open in $\mathscr M_X$, where $X$ is any smooth projective variety.
\begin{prop}
The stack $\catA_{\mu}^{(2)}$ is an algebraic stack that is locally of finite type over $\C$.
The morphism $(p_1,p_2) \colon \catA_{\mu}^{(2)} \to \catA_{\mu} \times \catA_{\mu}$ is of finite type.
\end{prop}
Given two $1$-morphisms $f_i\colon \mathscr X_i\ra \mathscr A_\mu$, consider the diagram of stacks
\be\label{HAproduct}
\begin{tikzcd}
\mathscr X_1\star \mathscr X_2\MySymb{dr} \arrow{r}{f}\arrow{d} &
\mathscr A_\mu^{(2)}\arrow{r}{p_3}\arrow{d}{(p_1,p_2)} & 
\mathscr A_\mu \\
\mathscr X_1\times \mathscr X_2\arrow[swap]{r}{f_1\times f_2} & 
\mathscr A_\mu\times \mathscr A_\mu &
\end{tikzcd}
\ee
where the square is cartesian.
Then one defines
\[
\left[f_1\colon \mathscr X_1\ra \mathscr A_\mu\right]\star \left[f_2\colon\mathscr X_2\ra \mathscr A_\mu\right] = 
\left[p_3\circ f\colon \mathscr X_1\star \mathscr X_2\ra \mathscr A_\mu\right].
\]
As a consequence of the previous proposition, the stack $\mathscr X_1 \star \mathscr X_2$ is algebraic and of finite type.
It also has affine geometric stabilisers, and thus defines an element of $H(\catA_{\mu})$.
The unit is given by the class $1 = [\Spec\C \to \catA_{\mu}]$ corresponding to the zero object $0 \in \catA_{\mu}$.
\begin{theorem}
The triple $(H(\catA_{\mu}),\star,1)$ defines a unital associative algebra.
\end{theorem}
\begin{proof}
The proof of \cite[Thm.~4.3]{Bri-Hall} goes through without change.
\end{proof}
Let $\Gamma\subset H^\ast(X,\mathbb Q)$ be the image of the Chern character map.
It is a finitely generated free abelian group.
The stack $\catA_{\mu}$ decomposes as a disjoint union into open and closed substacks 
\[
\catA_{\mu} = \coprod_{\gamma \in \Gamma} \catA_{\mu}^{(\gamma)}
\]
where $\catA_{\mu}^{(\gamma)}$ is the substack of objects of Chern character $\gamma$.
The Hall algebra is $\Gamma$-graded
\[
H(\catA_{\mu}) = \bigoplus_{\gamma \in \Gamma} H_{\gamma}(\catA_{\mu})
\]
where $H_{\gamma}(\catA_{\mu})$ is spanned by classes of maps $[\mathscr X \to \catA_{\mu}]$ that factor through $\catA_{\mu}^{(\gamma)} \subset \catA_{\mu}$.

\subsubsection{Integration morphism}
For simplicity, we denote a symbol $[T \to \Spec\C]$ in the Grothendieck group $K(\St/\C)$ by $[T]$.
This group has a natural commutative ring structure induced by the fibre product of stacks $[T] \cdot [U] = [T \times_{\C} U]$.
In turn, the Hall algebra has a natural structure of $K(\St/\C)$-module where the action is given by
\[
[T] \cdot [U \xrightarrow{f} \catA_{\mu}] = [T \times U \xrightarrow{f \circ \text{pr}_2} \catA_{\mu}].
\]
The ring $K(\St/\C)$ is obtained from the classical Grothendieck ring of varieties $K(\Var/\C)$ by localising at the classes of special algebraic groups \cite[Lemma 3.8]{Bri-Hall}.
There is a subring
\[
\Lambda = K(\Var/\C)[\L^{-1},(1+\L+\cdots+\L^k)^{-1} \colon k \geq 1] \subset K(\St/\C),
\]
where $\L = [\A^1]$ is the Lefschetz motive.
Note that $\Lambda$ does not contain $[B\C^{\ast}] = (\L-1)^{-1}$.

By \cite[Thm.~5.1]{Bri-Hall}, the $\Lambda$-submodule\footnote{The classes $1+\L+\cdots+\L^k$ of the projective spaces $\P^k$ have to be inverted in the definition of the regular Hall subalgebra; see the corrected version \cite{Bri} available on the arXiv.} 
\[
H^{\reg}(\catA_{\mu})\subset H(\catA_{\mu})
\]
generated by \emph{regular} elements is closed under the $\star$ product, where a regular element is an element in the span of the classes $[f \colon T \to \catA_{\mu}]$ where $T$ is a variety.
Moreover, the quotient 
\[
H^{\sc}(\mathscr A_\mu) = H^{\reg}(\mathscr A_\mu)/(\L-1)\cdot H^{\reg}(\mathscr A_\mu)
\]
of the regular subalgebra $H^{\reg}(\mathscr A_\mu)\subset H(\mathscr A_\mu)$ by the ideal generated by $\L-1$ is a commutative algebra called the \emph{semi-classical limit}.
It is equipped with an induced Poisson bracket
\begin{equation}\label{eq:Poisson_Bracket}
\Set{f,g} = \frac{f\star g-g\star f}{\L-1}
\end{equation}
with respect to $\star$, where $f,g \in H^{\sc}(\catA_{\mu})$.
The semi-classical limit is the domain of the integration map.

The codomain of the integration map is the \emph{Poisson torus}.
There are two versions depending on the choice of a sign $\sigma \in \{\pm 1\}$.
Either is defined as the $\Q$-vector space
\[
C_{\sigma}(X) = \bigoplus_{\gamma \in \Gamma} \Q\cdot c_{\gamma}
\]
generated by elements $\Set{c_{\gamma} \mid \gamma \in \Gamma}$ and equipped with a product given by the rule
\[
c_{\gamma_1}\star c_{\gamma_2} = \sigma^{\chi(\gamma_1,\gamma_2)}\cdot c_{\gamma_1+\gamma_2}.
\]
The Poisson torus is a Poisson algebra with respect to the bracket
\[
\Set{c_{\gamma_1},c_{\gamma_2}} = \sigma^{\chi(\gamma_1,\gamma_2)}\chi(\gamma_1,\gamma_2) \cdot c_{\gamma_1+\gamma_2}.
\]

\begin{remark}
Note that the product is commutative when $\sigma = +1$.
If $X$ is Calabi--Yau the Euler pairing is anti-symmetric, so the product is also commutative when $\sigma = -1$.
\end{remark}
There are two integration morphisms $H^{\sc}(\mathscr A_\mu)\ra C_{\sigma}(X)$ defined by 
\be\label{integrationmaps}
\left[f\colon Y\ra \mathscr A_\mu\right] \,\,\mapsto\,\,
\begin{cases}
\chi(Y) \cdot c_{\gamma}, & \textrm{if }\sigma = 1 \\
\chi(Y,f^\ast \nu) \cdot c_{\gamma}, & \textrm{if }\sigma = -1,
\end{cases}
\ee
whenever $f$ factors through the open and closed substack $\catA_{\mu}^{(\gamma)}\subset \catA_{\mu}$ for some $\gamma \in \Gamma$, and where $\nu$ denotes the Behrend weight on $\mathscr A_\mu$.
For simplicity, we denote these group homomorphisms by $I^E$ (when $\sigma=1$) and $I^B$ (when $\sigma=-1$) for Euler and Behrend respectively.

\begin{remark}\label{rem:Integration_Morphisms}
The result \cite[Thm.~5.2]{Bri-Hall} gives conditions which guarantee that $I^E$ and $I^B$ are morphisms of commutative or Poisson algebras.
We summarise these here.
\benum
\item If $X$ is any smooth projective $3$-fold, then $I^E$ is a morphism of commutative algebras.
It is not a morphism of Poisson algebras in general.
However, if $[f \colon Y \to \catA_{\mu}] \in H^{\sc}(\catA_{\mu})$ factors through $\catA_{\mu}^{(\gamma_0)}$ where $\gamma_0 \in \Gamma$ is the Chern character of a zero-dimensional object, then
\[
I^E\left(\{f,g\}\right) = \left\{I^E(f),I^E(g)\right\}
\]
holds for all $g \in H^{\sc}(\catA_{\mu})$.
This follows from the proof of \cite[Thm.~5.2]{Bri-Hall}, because Serre duality between a zero-dimensional object $Z$ and any object $E$ in $D^b(X)$ reduces to the Calabi--Yau condition $\Ext^i(Z,E) \cong \Ext^{3-i}(E,Z)$ for all $i \in \Z$.
\item If $X$ is Calabi--Yau, then both $I^E$ and $I^B$ are morphism of Poisson algebras.
\eenum
\end{remark}

Finally, we need a completed version of both the Hall algebra and the Poisson torus.
Toda constructs in \cite[Sec.~3.4]{Toda2} a completed Hall algebra
\[
\widehat H_{\#}(\mathscr A_{\mu}) = \prod_{\gamma \in \Gamma_{\#}} H_{\gamma}(\mathscr A_{\mu}) = \prod_{\gamma \in \Gamma_{\#}}H_{\gamma}(\Coh_{\leq 1}(X)[-1]),
\]
where $\Gamma_{\#} = \set{(0,0,-\beta',-n') \in \Gamma \mid \beta' \geq 0, n' \geq 0}$ can be seen as the set of Chern characters of objects in $\mathscr A_\mu$ belonging to $\Coh_{\leq 1}(X)[-1]$, cf.~\cite[Remark 3.11]{Toda2}.
Moreover, for each $(r,D)$ satisfying \eqref{coprime}, he constructs a bimodule over this algebra,
\[
\widehat H_{r,D}(\mathscr A_\mu) = \prod_{\gamma \in \Gamma_{r,D}}H_{\gamma}(\mathscr A_\mu),
\]
where $\Gamma_{r,D} \subset \Gamma$ is a suitably bounded subset\footnote{See \cite[Lemma~3.9]{Toda2}, which implies that the product and bracket extend to these completions.} of admissible Chern characters of the form $(r,D,-\beta,-n)$. 
There is a corresponding $\Lambda$-submodule of regular elements
\[
\widehat H_{r,D}^{\reg}(\mathscr A_\mu)\subset \widehat H_{r,D}(\mathscr A_\mu).
\]
The semi-classical limit $\widehat H_{r,D}^{\sc}(\mathscr A_\mu)$ is defined as the quotient of $\widehat H_{r,D}^{\reg}(\mathscr A_\mu)$ by the submodule generated by $\L-1$.
These are bimodules over the corresponding algebras $\widehat H_{\#}^{\sc}(\mathscr A_\mu)$ and $\widehat H_{\#}^{\reg}(\mathscr A_\mu)$ with respect to $\star$-multiplication on the left and on the right.

There are (Behrend weighted and unweighted) integration morphisms induced by \eqref{integrationmaps},
\begin{equation}\label{intmap}
\begin{split}
I_{r,D} &\colon \widehat H_{r,D}^{\sc}(\catA_{\mu}) \to \widehat{C}_{r,D}(\catA_{\mu}) = \prod_{\gamma \in \Gamma_{r,D}}C_{\gamma}(X), \\
I_{\#} &\colon \widehat H_{\#}^{\sc}(\catA_{\mu}) \to \widehat{C}_{\#}(\catA_{\mu}) = \prod_{\gamma \in \Gamma_{\#}} C_{\gamma}(X),
\end{split}
\end{equation}
to the completed Poisson tori.
These are compatible with the Poisson-bimodule structure of $\widehat H_{r,D}^{\sc}(\catA_{\mu})$ over $\widehat H_{\#}^{\sc}(\catA_{\mu})$ and $\widehat{C}_{r,D}(\catA_{\mu})$ over $\widehat{C}_{\#}(\catA_{\mu})$ in the following sense:
\begin{equation}\label{eq:Poisson_bimodule_morphism}
I_{r,D}\left(\{a,x\}\right) = \left\{ I_{\#}(a),I_{r,D}(x) \right\}
\end{equation}
for all $a \in \widehat H_{\#}^{\sc}(\catA_{\mu})$ and $x \in \widehat H_{r,D}^{\sc}(\catA_{\mu})$, where $\{a,-\} \colon \widehat H_{r,D}^{\sc}(\catA_{\mu}) \to \widehat H_{r,D}^{\sc}(\catA_{\mu})$.

In short, $I_{r,D}$ is a Poisson-bimodule morphism over $\widehat H_{\#}^{\sc}(\catA_{\mu})$ whereas $I_{\#}$ is a morphism of Poisson algebras.
We use \eqref{intmap} in Section \ref{DTPTcorrespondence} to obtain the $\mathcal F$-local DT/PT correspondence (Theorem \ref{thm:DTPT}) and to reprove equation~\eqref{GKformula} for stable sheaves (Theorem \ref{thm:gk}).

\section{The higher rank local DT/PT correspondence}\label{Sec:dtpt}
Let $X$ be a smooth projective $3$-fold and let $\mathcal{F} \in M_{\DT}(r,D)$ be a $\mu_{\omega}$-stable sheaf of homological dimension at most one, i.e., $\mathcal{F}$ satisfies $\lExt^i(\F,\O_X) = 0$ for $i\geq 2$.

In this section, we embed the Quot schemes $\Quot_X(\mathcal F)$ and $\Quot_X(\lExt^1(\mathcal F,\O_X))$ in suitable moduli spaces of torsion free sheaves and PT pairs, respectively.
In the Calabi--Yau case, we use these embeddings to define $\mathcal F$-local DT and PT invariants of $X$, representing the virtual contributions of $\mathcal F$ to the global invariants.\footnote{The definition makes sense for arbitrary $3$-folds, but the enumerative meaning of these numbers is less clear without the Calabi--Yau condition.}
We prove Theorem \ref{thm2} (resp.~Theorem \ref{thm3}) by applying the integration morphism $I^B$ (resp.~$I^E$) to an identity in the Hall algebra of $\catA_\mu$ (Proposition \ref{HallIdentity2}), which is the $\F$-local analogue of the global identity \cite[Thm.~1.2]{Toda2}.

\subsection{Embedding Quot schemes in DT and PT moduli spaces}

Recall that closed immersions of schemes are the proper \emph{monomorphisms} (in the categorical sense), and a monomorphism of schemes is a morphism $a \colon Y \to Z$ such that the induced natural transformation $\Hom(-,Y) \to \Hom(-,Z)$ is injective.
If $Y$ and $Z$ are finite type schemes over an algebraically closed field, then $a$ is a closed immersion if and only if it is proper, injective on closed points, and injective on tangent spaces at closed points.
We will use the latter characterisation in Proposition \ref{DTembedding} and the former characterisation in Proposition \ref{PTembedding}.

\begin{prop}\label{DTembedding}
Taking the kernel of a surjection $\mathcal F\onto Q$ defines a closed immersion 
\[
\phi_{\F} \colon \Quot_X(\mathcal F) \into M_{\DT}(r,D).
\]
\end{prop}

\begin{proof}
Let $B$ be a scheme, $\pi_X\colon X\times B\ra X$ the projection, and $\pi_X^\ast \F\onto \mathscr Q$ a flat family of zero-dimensional quotients of $\F$. Then the kernel $\mathcal K\subset \pi_X^\ast \F$ is $B$-flat and defines a family of torsion free $\mu_\omega$-stable sheaves on $X$. Therefore the association
\[
[\mathcal F\onto Q]\mapsto \ker(\mathcal F\onto Q)
\]
is a morphism to the moduli stack ${\mathcal M}_{\DT}(r,D)$, and composing with the natural morphism $p \colon {\mathcal M}_{\DT}(r,D) \to M_{\DT}(r,D)$ to the coarse moduli space defines $\phi_{\F}$.
Note that the morphism $p$ is a $\mathbb{G}_{m}$-gerbe by the coprimality condition~\eqref{coprime}.
Since the Quot scheme is proper and $M_{\DT}(r,D)$ is separated, the morphism is proper.

Next, we verify that $\phi_{\mathcal{F}}$ is injective on tangent spaces.
Let $q = [\mathcal F \onto Q]$ be a closed point of $\Quot_X(\mathcal F)$, and let $K \subset \mathcal F$ be the kernel of the surjection.
Note that $K$ is stable, hence simple, and so $\hom(K,K) = 1$.
The tangent space to $\Quot_X(\mathcal F)$ at $q$ is $\Hom(K,Q)$.
To establish injectivity of $\phi_{\mathcal F}$ on tangent spaces, it suffices to show that in the exact sequence
\[
0\ra \Hom(K,K)\xrightarrow{i} \Hom(K,\mathcal F)\xrightarrow{u} \Hom(K,Q)\xrightarrow{\dd\phi_{\mathcal F}} \Ext^1(K,K),
\]
where $\dd\phi_{\mathcal F}$ is the tangent map to $\phi_{\mathcal F}$ at $q$, one has $u = 0$.
In the exact sequence
\[
\Hom(Q,\mathcal F)\ra \Hom(\mathcal F,\mathcal F)\ra \Hom(K,\mathcal F)\ra \Ext^1(Q,\mathcal F),
\]
the two outer terms vanish by Remark~\ref{rmk:vanishingext}. By stability of $\mathcal F$, this vanishing implies that $\hom(K,\mathcal F) = 1$.
Therefore $i$ is an isomorphism, and $u = 0$ as required.

Finally, we prove that $\phi_{\mathcal F}$ is injective on closed points.
Let $q_{i} = [\mathcal{F} \onto Q_i]$ for $i = 1,2$ be two closed points, and assume that $\phi_{\mathcal{F}}(q_1) = \phi_{\mathcal{F}}(q_2)$.
Thus, their kernels are isomorphic as sheaves, say via $\alpha \colon K_1\, \widetilde{\to}\, K_2$.
By the previous argument we have $\hom(K_i,\mathcal{F}) = 1$, implying that the embeddings of the kernels into $\mathcal{F}$, which we denote by $\iota_{i} \colon K_{i} \into \mathcal{F}$, are unique up to scaling.
It now follows, scaling the isomorphism $\alpha$ if necessary, that the diagram 
\[
\begin{tikzcd}
    K_1 \arrow[hook]{r}{\iota_1} \arrow[swap]{d}{\alpha} & \mathcal{F} \arrow[equal]{d} \\
    K_2 \arrow[hook]{r}{\iota_2} & \mathcal{F}
\end{tikzcd}
\]
commutes.
Thus $q_1 = q_2$ in $\Quot_X(\mathcal{F})$, proving that $\phi_{\mathcal{F}}$ is a closed immersion.
\end{proof}

\begin{remark}
As the proof shows, the initial assumption on the homological dimension of $\F$ is not needed. It will be needed in Proposition \ref{PTembedding} (the ``PT side'').
\end{remark}

We now move to the PT side.
First, we construct a map
\[
\Quot_X(\lExt^1(\mathcal F,\O_X))\ra M_{\PT}(r,D)
\]
on the level of $\C$-valued points.
Consider a surjection
\[
t \colon \lExt^1(\mathcal F,\O_X) \onto Q,
\]
and recall the following identifications, induced by the derived dualising functor,
\be\label{homs617}
\Hom(\mathcal F^{\vee},Q[-1]) = \Hom(Q^{\textrm{D}}[-2],\mathcal F) = \Ext^1(Q^{\textrm{D}}[-1],\mathcal F).
\ee
We interpret $t$ as an element of the first Hom-space by precomposing its shift
\[
\bar{t} \colon \mathcal F^{\vee} \to \lExt^1(\mathcal F,\O_X)[-1] \xrightarrow{t[-1]} Q[-1],
\]
and we associate to $t$ the extension
\be\label{triangle10482}
\mathcal F \into J^\bullet \onto Q^{\textrm{D}}[-1]
\ee
in $\mathscr{A}_{\mu}$ corresponding to $\bar{t}$ under \eqref{homs617}. 
Note that $\rk J^\bullet = \rk \mathcal F = r$ and $c_1(J^\bullet) = c_1(\mathcal F) = D$.
It is clear that $J^\bullet$ defines a pre-PT pair based at $\mathcal F$.
To see that $J^\bullet$ is in fact a PT pair, we dualise again by applying $\lHom(-,\O_X)$ to the defining triangle \eqref{triangle10482}.
We find
\[
\cdots \to \lExt^1(\mathcal F,\O_X) \xrightarrow{t} Q = \lExt^3(Q^{\textrm{D}},\O_X) \to \lExt^2(J^\bullet,\O_X) \to 0
\]
where the last zero is $\lExt^2(\mathcal F,\O_X) = 0$.
Here, we have $t = H^1(\bar{t})$ since the derived dual is an involution.
Thus the surjectivity of the morphism $t$ is equivalent to the vanishing of $\lExt^2(J^\bullet,\O_X)$.
In turn, by Lemma~\ref{lem:PrePT_is_PT}, this means that $J^\bullet$ is a PT pair.

\begin{remark}
Conversely, any $\mathcal F$-local PT pair, consisting of an exact triangle of the form \eqref{triangle10482}, gives rise to a surjection $\lExt^1(\mathcal F,\O_X) \onto Q$ by applying $\lHom(-,\O_X)$.
\end{remark}

To extend this association to families, we make use of the following result.
\begin{lemma}\label{TopExt}
Let $B$ be a scheme, and let $\mathscr Q\in \Coh(X\times B)$ be a $B$-flat family of zero-dimensional sheaves on a smooth scheme $X$ of dimension $d$. Then one has
\[
\lExt^i_{X\times B}(\mathscr Q,\O_{X\times B}) = 0, \qquad i\neq d.
\]
Moreover, the base change map
\[
\lExt^d_{X\times B}(\mathscr Q,\mathscr O_{X\times B})\otimes_{\O_B}k(b) \ra \lExt^d_X(\mathscr Q_b,\mathscr O_X)
\]
is an isomorphism for all $b\in B$.
\end{lemma}

\begin{proof}
The vanishing follows from \cite[Theorem 1.10]{PicScheme}.
The base change property follows from \cite[Theorem 1.9]{PicScheme}.
\end{proof}

\begin{prop} \label{PTembedding}
The association $t\mapsto J^\bullet$ extends to a closed immersion
\[
\psi_{\mathcal F} \colon \Quot(\lExt^1(\mathcal F,\O_X)) \into M_{\PT}(r,D).
\]
\end{prop}

\begin{proof}
Let $B$ be a base scheme, and let $\pi_X\colon X\times B\ra X$ denote the natural projection.
Let $\mathscr Q \in \Coh(X \times B)$ be a $B$-flat family of zero-dimensional sheaves on $X$ receiving a surjection
\begin{equation*}
t \colon \pi_X^*\lExt^1(\mathcal F,\O_X) \onto \mathscr Q.
\end{equation*}
In particular, $\mathscr Q \in \Perf(X \times B)$ is a perfect object in the derived category of $X \times B$.
Pulling back the triangle $H^0(\mathcal F^\vee)\ra \mathcal F^\vee \ra \lExt^1(\mathcal F,\O_X)[-1]$ on $X$ to $X \times B$ yields the exact triangle
\be\label{triangle7163}
\pi_X^*H^0(\mathcal F^{\vee}) \to \pi_X^*(\mathcal F^{\vee}) \to \pi_X^*\lExt^1(\mathcal F,\O_X)[-1],
\ee
and precomposing $t[-1]$ yields the morphism $\bar{t} \colon \pi_X^*(\mathcal F^{\vee}) \to \mathscr Q[-1]$.
Since all objects we consider are perfect complexes on $X \times B$, so are their derived duals.
We deduce
\begin{equation*}
\pi_X^*(\mathcal F^{\vee}) = (\pi_X^*\mathcal F)^{\vee}.
\end{equation*}
More generally, for complexes $E,F \in \Perf(X \times B)$ we have a natural isomorphism
\[
{\textbf R}\lHom(E,F) \cong {\textbf R}\lHom(F^{\vee},E^{\vee}).
\]
By Lemma~\ref{TopExt}, we have a natural isomorphism of perfect complexes on $X \times B$
\begin{equation*}
{\textbf R} \lHom\left(\pi_X^*(\mathcal F^{\vee}),\mathscr Q[-1]\right) \cong {\textbf R} \lHom\left(\mathscr Q^{\textrm{D}}[-2],\pi_X^*\mathcal F\right).
\end{equation*}
Taking derived global sections yields an isomorphism of complexes of vector spaces, and further taking cohomology yields a natural isomorphism of $\C$-linear Hom-spaces
\be\label{gnagna311}
\Hom_{X \times B}\left(\pi_X^*(\mathcal F^{\vee}),\mathscr Q[-1]\right) \cong \Ext^1_{X \times B}\left({\mathscr Q}^{\textrm{D}}[-1],\pi_X^*\mathcal F\right).
\ee
We write the image of $\bar{t}$ under this identification as an extension
\be\label{Jfamily}
\pi_X^*\mathcal F \to J^{\bullet} \to \mathscr Q^{\textrm{D}}[-1].
\ee
We claim that $J^{\bullet}$ is a family of PT pairs parametrised by $B$ (cf.~Definition \ref{def:familyPT}).

First, taking the derived fibre of the triangle \eqref{Jfamily} shows that $J_b^\bullet = {\textbf L} i_b^*(J^{\bullet})$ is a pre-PT pair on $X_b$ 
for all closed points $b \in B$, where $X_b = X\times \set{b}$ and $i_b\colon X_b \into X\times B$ is the natural closed immersion.
Second, $J^{\bullet}$ is a perfect complex because $\pi_X^\ast \mathcal F$ and $\mathscr Q^{\textrm{D}}[-1]$ are.
Thus $J^{\bullet}$ defines a family of pre-PT pairs based at $\mathcal F$ in the sense of Definition~\ref{def:familyPT}.

To see that each derived fibre $J^{\bullet}_b$ is a PT pair, recall that for a perfect complex the operations of taking the derived fibre and taking the derived dual commute.
In other words, there is a \emph{canonical} isomorphism
\[
{\textbf L} i_b^*(E^{\vee}) \cong ({\textbf L} i_b^*E)^{\vee}
\]
for all $b \in B$, where $E \in \Perf(X \times B)$ is a perfect complex.
As a consequence, the diagram
\begin{equation*}
\begin{tikzcd}[row sep=large]
\Hom_{X \times B}\left(\pi_X^*\mathcal F^{\vee},{\mathscr Q}[-1]\right) \arrow[r,"(-)^{\vee}"]\arrow[d,"{\textbf L} i_b^*"] & 
\Hom_{X \times B}\left({\mathscr Q}^{\textrm{D}}[-2], \pi_X^*\mathcal F \right) \arrow[d,"{\textbf L} i_b^*"] \\
\Hom_{X_b}\left(\mathcal F^{\vee},{\mathscr Q}_b[-1]\right) \arrow[r,"(-)^{\vee}"] &
\Hom_{X_b}\left({\mathscr Q}^{\textrm{D}}_b[-2],\mathcal F\right)
\end{tikzcd}
\end{equation*}
commutes.\footnote{Here we are implicitly using Lemma \ref{TopExt}, saying that dualising commutes with base change for a flat family of zero-dimensional sheaves.
More generally, dualising commutes with any base change for perfect objects.}
In words, given the morphism $\bar{t} \colon \pi_X^*(\mathcal F^{\vee}) \to {\mathscr Q}[-1]$, we obtain our family of pre-PT pairs $J^{\bullet}$ on $X \times B$ by dualising, and taking its derived fibre $J^{\bullet}_b = {\textbf L} i_b^*(J^{\bullet})$ yields a complex that is canonically isomorphic to the pre-PT pair obtained by first restricting $t$ to the derived fibre
\[ 
\bar{t}_b = {\textbf L} i_b^*(\bar{t}) \colon \mathcal F^{\vee} \to \lExt^1(\mathcal F,\O_X)[-1] \xrightarrow{t_b[-1]} {\mathscr Q}_b[-1]
\]
and then dualising.
By taking the derived fibre of the triangle \eqref{Jfamily} and running the argument for a $\C$-valued point, it follows that $\lExt^2(J^{\bullet}_b,\O_{X_b}) = 0$ because each map $t_b$ is surjective.
Thus $J^{\bullet}$ defines a family of PT pairs by Lemma~\ref{lem:PrePT_is_PT}.
We have obtained a morphism to the moduli stack $\mathcal{M}_{\PT}(r,D)$, and composing with $\mathcal{M}_{\PT}(r,D) \to M_{\PT}(r,D)$, defines $\psi_{\mathcal F}$.

Note that $\psi_{\mathcal F}$ is proper because its domain is proper and the PT moduli space is separated.
To prove that it is a closed immersion, it is enough to show it is injective on $B$-valued points.
Let $t$ and $u$ be two families of surjections with (the same) target ${\mathscr Q}$, and assume that they give rise to the same family of PT pairs $J^\bullet$.
Then, by \eqref{gnagna311}, we conclude that $\bar t= t[-1]\circ g = u[-1]\circ g=\bar u$, where
\[
g \colon \pi_X^*(\mathcal F^{\vee}) \to \pi_X^*\lExt^1(\mathcal F,\O_X)[-1]
\]
is the morphism appearing in \eqref{triangle7163}.
But the natural map
\[
(-)\circ g\colon \Hom\left(\pi_X^*\lExt^1(\mathcal F,\O_X)[-1],\mathscr Q[-1]\right)\ra \Hom\left(\pi_X^*(\mathcal F^{\vee}),\mathscr Q[-1]\right)
\]
is injective because $\Hom(\pi_X^\ast H^0(\mathcal F^\vee)[1],\mathscr Q[-1]) = \Ext^{-2}(\pi_X^*H^0(\mathcal F^\vee),{\mathscr Q}) = 0$, i.e., its kernel vanishes.
It follows that $t = u$.
This completes the proof.
\end{proof}

See Section \ref{UniversalEmbedding}, and in particular Proposition \ref{prop:UnivPTimmersion}, for a `universal' closed immersion generalising Proposition \ref{PTembedding}.

\subsection{The $\mathcal F$-local Hall algebra identity}\label{DTPT_identity}
In this section, we prove an $\F$-local analogue of Toda's Hall algebra identity in \cite[Lemma~3.16]{Toda2} which gives rise to the higher rank DT/PT correspondence by applying the integration morphisms $I^B$.

We introduce the Hall algebra elements.
Let $\Coh_0(X)[-1]$ be the shift of the category of zero-dimensional coherent sheaves on $X$, which we denote by $\mathcal C_{\infty}$ to be consistent with the notation of \cite{Toda2}.
The moduli stack of objects in ${\mathcal C}_{\infty}$ is an open substack $\mathcal C_{\infty} \subset \mathcal A_{\mu}$.
We obtain an element
\[
\delta(\mathcal C_\infty) = \left[\mathcal C_\infty \hookrightarrow \mathcal A_\mu\right]\in \widehat H_{\#}(\mathscr A_\mu).
\]

Let $p\colon \mathcal M_{\DT}(r,D)\ra M_{\DT}(r,D)$ denote the natural morphism from the moduli stack of DT objects to its coarse moduli space.
Consider the cartesian diagram
\[
\begin{tikzcd}
\mathcal Q_X(\F)\MySymb{dr} \ar[hook]{r}{\iota_{\F}}\ar[swap]{d}{p'} & \mathcal M_{\DT}(r,D)\ar{d}{p}\arrow[hook]{r}{\textrm{open}} & \catA_\mu\\
\Quot_X(\F) \ar[hook,swap]{r}{\phi_{\F}} & M_{\DT}(r,D) &
\end{tikzcd}
\]
defining $\mathcal Q_X(\F)$. The map $\iota_{\F}$ is a closed immersion by Proposition \ref{DTembedding} and base change.
Hence $\mathcal Q_X(\F)$ defines a locally closed substack of $\catA_\mu$. A similar picture holds on the PT side by replacing $\F$ by $\lExt^1(\F,\O_X)$.

For the sake of brevity, we rename these objects 
\[
\mathcal Q_{\DT}^{\F} = \mathcal Q_X(\F),\quad \mathcal Q_{\PT}^{\F} = \mathcal Q_X(\lExt^1(\F,\O_X)).
\]
We now obtain Hall algebra elements
\[
    \delta_{\DT}^{\F} = \left[\mathcal Q_{\DT}^{\F}\ra \catA_\mu\right],\quad 
    \delta_{\PT}^{\F} = \left[\mathcal Q_{\PT}^{\F}\ra \catA_\mu\right]
\]
in $\widehat H_{r,D}(\catA_\mu)$.
By base change, the morphism $p'$ and its analogue for $\lExt^1(\F,\O_X)$ are (trivial) $\G_m$-gerbes (cf.~Remark~\ref{rem:stacks_gerbes}).
Thus it follows that the elements
\[
\overline \delta_{\DT}^{\F} = (\L-1)\cdot \delta_{\DT}^{\F},\quad \overline \delta_{\PT}^{\F} = (\L-1)\cdot \delta_{\PT}^{\F}
\]
lie in the regular submodule $\widehat H_{r,D}^{\reg}(\mathscr A_\mu)$. 
Projecting to the semiclassical limit yields elements
\be\label{Main_HA_elements}
\overline \delta_{\DT}^{\F},\,\,\,\overline \delta_{\PT}^{\F} \,\in\,
\widehat H_{r,D}^{\sc}(\mathscr A_\mu).
\ee

Let us form the $\mathscr A_\mu$-stacks 
\[
\mathcal Q_{\DT}^{\mathcal F}\star \mathcal C_\infty,\quad \mathcal C_\infty\star \mathcal Q_{\PT}^{\mathcal F}
\]
via the pullback construction described in \eqref{HAproduct}. For a scheme $B$, a $B$-valued point of $\mathcal Q_{\DT}^{\mathcal F}\star \mathcal C_\infty$ is an exact triangle $E_1 \to E \to E_3$ in $\Perf(X \times B)$ such that $E_1$, $E_3 \in \catA_{\mu}(B)$, $E_1$ is a $B$-valued point of $\mathcal Q_{\DT}^{\mathcal F}$, and $E_3$ is a $B$-valued point of $\mathcal C_{\infty}$.
Similarly, a $B$-valued point of $\mathcal C_\infty\star \mathcal Q_{\PT}^{\mathcal F}$ is an exact triangle $E_1 \to E \to E_3$ in $\Perf(X \times B)$ where $E_1$ is a $B$-valued point of $\mathcal C_\infty$ and $E_3$ is a $B$-valued point of $\mathcal Q_{\PT}^{\mathcal F}$.

\begin{lemma}\label{Closed_Points_13}
There is an equivalence at the level of $\C$-valued points
\[
\left(\mathcal Q_{\DT}^{\mathcal F}\star \mathcal C_\infty\right)(\C) = \left(\mathcal C_\infty\star \mathcal Q_{\PT}^{\mathcal F}\right)(\C).
\]
\end{lemma}

\begin{proof}
Let $E$ be a $\C$-valued point of $\mathcal C_\infty\star \mathcal Q_{\PT}^{\mathcal F}$. Then we can decompose $E$ as an extension of an object $J^\bullet \in \mathcal Q_{\PT}^{\mathcal F}$ by a (shifted) zero-dimensional object $Q[-1]\in\mathcal C_\infty$.
We obtain the following diagram
\[
\begin{tikzcd}
	H^0(E) \arrow[r,hook] \arrow[d,"f"] & E \arrow[r,two heads] \arrow[d,two heads] & H^1(E)[-1] \arrow[d,"g"]\\
    \mathcal F \arrow[r,hook] & J^\bullet \arrow[r,two heads] & P[-1]
\end{tikzcd}
\]
in $\mathscr{A}_{\mu}$, where $P\in \Coh_0(X)$.
The snake lemma in $\catA_{\mu}$ induces a four-term exact sequence
\[
0 \to \ker(f) \to Q[-1] \to \ker(g) \to \coker(f) \to 0
\]
and implies that $g$ is surjective.
Since $\Coh_{\leq 1}(X)[-1] \subset \mathscr{A}_{\mu}$ is closed under subobjects, extensions, and quotients, and $Q$ is zero-dimensional, we deduce that the above exact sequence lies in $\mathcal C_\infty$ entirely.
But there are no morphisms in negative degree, so $\coker(f) = 0$ since $\mathcal F \in \Coh_{\mu}(X)$ is a sheaf.
We obtain the exact sequence
\[
	0 \to H^0(E) \to \F \to \ker(f)[1] \to 0
\]
in $\Coh(X)$, proving that $H^0(E) \in \mathcal Q_{\DT}^{\mathcal F}$ as claimed.

Conversely, let $E$ be an extension of a (shifted) zero-dimensional object $P[-1]$ and an object $K \in \mathcal Q_{\DT}^{\mathcal F}$.
We obtain the diagram
\[
\begin{tikzcd}
	Q[-1] \arrow[r,hook] \arrow[d,"f'"] & K \arrow[r,two heads] \arrow[d,hook] & \mathcal F \arrow[d,"g'"]\\
    S[-1] \arrow[r,hook] & E \arrow[r,two heads] & J^\bullet
\end{tikzcd}
\]
in $\mathscr{A}_{\mu}$, where $S[-1] \subset E$ is the largest subobject of $E$ in $\Coh_0(X)[-1]$; this object exists since $\catA_{\mu}$ is noetherian and $\mathcal C_\infty \subset \catA_{\mu}$ is closed under extensions and quotients.
In particular, it follows that $\Hom(\mathcal{C}_{\infty},J^{\bullet}) = 0$.
The snake lemma in $\mathscr{A}_{\mu}$ induces a four-term exact sequence
\[
0 \to \ker(g') \to \coker(f') \to P[-1] \to \coker(g') \to 0
\]
and implies that $f'$ is injective.
As before, we deduce that the above exact sequence lies in $\mathcal C_\infty$ entirely.
By assumption $\mathcal F$ is both a DT and PT object, hence $\Hom(\mathcal C_\infty,\mathcal F) = 0$ and so $\ker(g') = 0$.
We obtain the exact sequence
\[
	0 \to \mathcal F \to J^{\bullet} \to \coker(g') \to 0
\]
in $\mathscr{A}_{\mu}$, proving that $J^\bullet \in \mathcal Q_{\PT}^{\mathcal F}$.
The constructions are clearly inverse to each other.
\end{proof}

We now refine the above identification to a Hall algebra identity. 

\begin{prop}\label{HallIdentity2}
In $\widehat H_{r,D}(\mathscr A_\mu)$, one has the identity
\begin{equation}\label{eq:Hall_Identity_Flocal}
\delta^{\F}_{\DT}\star \delta(\mathcal C_\infty) = \delta(\mathcal C_\infty) \star \delta^{\F}_{\PT}.
\end{equation}
\end{prop}
\begin{proof}
Recall the open immersion of stacks $\mathscr B_{\mu} \subset \mathscr A_{\mu}$.
The key Hall algebra identity proven by Toda relies on the existence of geometric bijections
\be\label{amustacks}
\begin{tikzcd}
\mathcal M_{\DT}(r,D)\star \mathcal C_\infty \arrow{dr} & & \mathcal C_\infty \star \mathcal M_{\PT}(r,D)\arrow{dl} \\
& \widetilde{\mathscr B}_{\mu} & 
\end{tikzcd}
\ee
where $\widetilde{\mathscr B}_{\mu}$ is the open substack of $\mathscr B_\mu$ parametrising objects $E\in\mathscr B_\mu$ such that $H^1(E)\in \Coh_0(X)$. 
Pulling back the closed immersions $\phi_{\mathcal F}$ and $\psi_{\mathcal F}$, the stacks $\mathcal Q_{\DT}^{\mathcal F}\star \mathcal C_\infty$ and $\mathcal C_\infty\star \mathcal Q_{\PT}^{\mathcal F}$ embed as closed substacks of the corresponding $\widetilde{\mathscr B}_{\mu}$-stacks \eqref{amustacks}.
Recall that the class of an $\catA_{\mu}$-stack $S \to \catA_{\mu}$ in the Hall algebra is equal to the class of its reduction,
\[
\left[S_{\text{red}} \to \catA_{\mu} \right] = \left[S \to \catA_{\mu} \right],
\]
via the geometric bijection $S_{\text{red}} \to S$ of $\catA_{\mu}$-stacks.
Thus we may assume that $\mathcal Q_{\DT}^{\mathcal F}\star \mathcal C_\infty$ and $\mathcal C_\infty\star \mathcal Q_{\PT}^{\mathcal F}$ are reduced.
Let $\mathscr Z_{\DT}$ and $\mathscr Z_{\PT}$ denote their stack-theoretic images in $\widetilde{\catB}_{\mu}$, which are reduced closed substacks.
\[
\begin{tikzcd}
{\mathscr Z}_{\DT} \arrow[swap,hook]{dr} & & {\mathscr Z}_{\PT} \arrow[left hook->]{dl} \\
& \widetilde{\mathscr B}_{\mu} & 
\end{tikzcd}
\]
We claim that ${\mathscr Z}_{\DT} = {\mathscr Z}_{\PT}$ as substacks.
This establishes the identity~\eqref{eq:Hall_Identity_Flocal} because, crucially, we have $\delta_{\DT}^{\mathcal{F}} \star \delta(\mathcal{C}_{\infty}) = [\mathscr{Z}_{\DT} \to \catA_{\mu}]$ in $\widehat{H}_{r,D}(\catA_{\mu})$ by the geometric bijection $\mathcal{Q}_{\DT}^{\F} \star \mathcal{C}_{\infty} \to \mathscr{Z}_{\DT}$ over $\widetilde{\mathscr B}_{\mu} \subset \catA_{\mu}$; similarly, we have the identity $\delta(\mathcal{C}_{\infty}) \star \delta_{\PT}^{\mathcal{F}} = [\mathscr{Z}_{\PT} \to \catA_{\mu}]$.

To prove this claim, recall that the stack $\catA_{\mu}$ is locally of finite type.
Since $\catA_{\mu} \subset D^b(X)$ is the heart of a bounded t-structure, it follows that the stack $\catA_{\mu}$ has affine geometric stabilisers \cite[Lemma~2.3.9]{SjoerdThesis}.
By a result of Kresch \cite[\textsection~4.5]{kreschcycle}, it is locally a global quotient stack $[V/G]$ where $V$ is a variety and $G$ is a linear algebraic group.
Thus locally ${\mathscr Z}_{\DT}$ and ${\mathscr Z}_{\PT}$ correspond to $G$-invariant closed subvarieties of $V$.
But the $\C$-valued points of these (reduced) closed subvarieties coincide by Lemma~\ref{Closed_Points_13}.
Thus ${\mathscr Z}_{\DT} = {\mathscr Z}_{\PT}$ as claimed. 
\end{proof}

\subsection{The $\mathcal F$-local DT/PT correspondence}\label{DTPTcorrespondence}
In this section, we prove Theorem \ref{thm2} and Theorem \ref{thm3}.
For $n\geq 0$, we define
\begin{align*}
\DDT_{\mathcal F,n} &= \chi(\Quot_X(\mathcal F,n),\nu_{\DT}) \\
\PPT_{\mathcal F,n} &= \chi(\Quot_X(\lExt^1(\mathcal F,\O_X),n),\nu_{\PT})
\end{align*}
where the Behrend weights come from the full DT and PT moduli spaces and are restricted via the closed immersions of Propositions \ref{DTembedding} and \ref{PTembedding}. 
We form the generating functions
\[
\DDT_{\mathcal F}(q) = \sum_{n\geq 0}\DDT_{\mathcal F,n}q^n,\quad \PPT_{\mathcal F}(q) = \sum_{n\geq 0} \PPT_{\mathcal F,n}q^n
\]
in the completed Poisson torus $\widehat C_{r,D}^{\sigma}(X)= \prod_{\gamma \in \Gamma_{r,D}}C_{\gamma}(X)$, where we use the shorthand $q = c_{(0,0,0,1)}$.
In the Calabi--Yau case, these series can be interpreted as the ``contribution'' of $\F$ to the global DT and PT invariants. 

Recall the Hall algebra elements \eqref{Main_HA_elements}.
Note that 
\be\label{Integration918}
I^B\left(\overline\delta_{\DT}^{\F}\right) = -\DDT_{\F}(q^{-1})c_{\gamma},\quad I^B\left(\overline\delta_{\PT}^{\F}\right) = -\PPT_{\F}(q^{-1}) c_{\gamma},
\ee
where $I^B$ is the Behrend weighted version of the map $I_{r,D}$ from \eqref{intmap}, and $\gamma = \ch(\mathcal{F})$.
The minus sign is a consequence of property \eqref{ChiSquare} of the Behrend function, taking into account that the moduli stacks are $\G_m$-gerbes over the coarse moduli spaces of DT and PT objects.

\begin{remark}
The Behrend weights $\nu_{\DT}$ and $\nu_{\PT}$ are not unrelated. Indeed, up to the $\G_m$-gerbes
\[
\mathcal M_{\DT}(r,D) \ra  M_{\DT}(r,D),\quad \mathcal M_{\PT}(r,D) \ra  M_{\PT}(r,D),
\]
they are both restrictions of the Behrend function of Lieblich's moduli stack $\mathscr M_X$, studied in \cite{Lieblich1}, along their respective open immersions $\mathcal M_{\DT}(r,D) \subset \mathscr{M}_X$ and $\mathcal M_{\PT}(r,D) \subset \mathscr{M}_X$.
\end{remark}

\begin{theorem}\label{thm:DTPT}
Let $X$ be a Calabi--Yau 3-fold, and let $\F$ be a $\mu_{\omega}$-stable sheaf of rank $r$ and homological dimension at most one.
There is an equality of generating series
\begin{equation}\label{FlocalDTPT}
\DDT_{\mathcal F}(q) = \mathsf M((-1)^rq)^{r\chi(X)}\cdot \PPT_{\mathcal F}(q).
\end{equation}
\end{theorem}

\begin{proof}
The equality follows by applying the integration morphism $I_{r,D}$ from \eqref{intmap} taken with $\sigma = -1$, precisely as in the proof of \cite[Thm.~3.17]{Toda2}, after replacing the Hall algebra identity of \cite[Lemma~3.16]{Toda2} by the identity of equation~\eqref{eq:Hall_Identity_Flocal}.
We give the argument in full.

By Proposition \ref{HallIdentity2}, we obtain the identity
\[
\delta^{\mathcal F}_{\DT} =
\delta(\catC_{\infty}) \star \delta^{\mathcal F}_{\PT} \star \delta(\catC_{\infty})^{-1}\,\in \,\widehat{H}_{r,D}(\catA_{\mu}).
\]
Since $\Coh_0(X)$ is artinian, we have a well-defined logarithm  $\epsilon(\catC_{\infty}) = \log \delta(\catC_{\infty}) \in \widehat{H}_{\#}(\catA_{\mu})$.
We obtain in $\widehat{H}_{r,D}(\catA_{\mu})$ the identity
\[
\delta^{\mathcal F}_{\DT} =
\exp\left(\epsilon(\catC_{\infty})\right) \star \delta^{\mathcal F}_{\PT} \star \exp\left(-\epsilon(\catC_{\infty})\right).
\]
The $\widehat{H}_{\#}(\catA_{\mu})$-bimodule structure of $\widehat{H}_{r,D}(\catA_{\mu})$ induces an adjoint action of $a \in \widehat{H}_{\#}(\catA_{\mu})$ on $\widehat{H}_{r,D}(\catA_{\mu})$ via the equation
\[
\Ad(a) \circ x = a \star x - x \star a \colon \widehat{H}_{r,D}(\catA_{\mu}) \to \widehat{H}_{r,D}(\catA_{\mu}).
\]
By the Baker--Campbell--Hausdorff formula the above equation becomes
\[
\delta^{\mathcal F}_{\DT} =
\exp\bigl(\Ad(\epsilon(\catC_{\infty}))\bigr) \circ \delta^{\mathcal F}_{\PT}
\]
in $\widehat{H}_{r,D}(\catA_{\mu})$.
Multiplying both sides of the equation by $\L-1$ and projecting the resulting equation in $\widehat{H}^{\reg}_{r,D}(\catA_{\mu})$ to the semi-classical quotient $\widehat{H}^{\sc}_{r,D}(\catA_{\mu})$, we obtain the identity
\[
\overline{\delta}^{\mathcal F}_{\DT}  =
\exp\bigl(\Ad^{\sc}(\overline{\epsilon}(\catC_{\infty}))\bigr) \circ \overline{\delta}^{\mathcal F}_{\PT} ~\in~ \widehat{H}^{\sc}_{r,D}(\catA_{\mu}),
\]
where we have written the adjoint action of $a \in \widehat{H}^{\sc}_{\#}(\catA_{\mu})$ as
\[
\Ad^{\sc}(a) \circ x \defeq \{a,x\} \colon \widehat{H}^{\sc}_{r,D}(\catA_{\mu}) \to \widehat{H}^{\sc}_{r,D}(\catA_{\mu}),
\]
in terms of the Poisson bracket of equation~\eqref{eq:Poisson_Bracket}, and we have applied Joyce's No-Poles Theorem which states that $\overline{\epsilon}(\catC_{\infty}) = (\L - 1)\epsilon(\catC_{\infty}) \in \widehat{H}^{\reg}_{\#}(\catA_{\mu})$; cf.~\cite[Thm.~3.12]{Toda2}.

Applying the integration morphism $I_{r,D}$, and using the Euler pairing computation
\begin{align*}
\chi\bigl((0,0,-\beta',-n'),(r,D,-\beta,-n) \bigr) = rn'-D\beta',
\end{align*}
along with the identities \eqref{Integration918}, we obtain the formula
\begin{equation}\label{almostthere}
    \DDT_{\mathcal F}(q) = \exp\left( \sum_{m > 0} (-1)^{rm-1} rm \cdot \mathsf N_{m,0} q^m\right)\cdot \PPT_{\mathcal F}(q),
\end{equation}
after formally sending $q^{-1} \mapsto q$.
Here the ``$\mathsf N$-invariants'' $\mathsf N_{m,0} \in \Q$ count semistable zero-dimensional sheaves $E$ with $\chi(E)=m$, and are defined by the relation
\[
I_{\#}\left(\overline{\epsilon}(\catC_{\infty})\right) = - \sum_{m \geq 0} \mathsf N_{m,0} (q^{-1})^m.
\]
See \cite[Sec.~3.6]{Toda2} and the references therein for more details.
Using the rank one identity
\[
\exp\left( \sum_{m > 0} (-1)^{m-1} m \cdot \mathsf N_{m,0} q^m\right) = \mathsf M(-q)^{\chi(X)}
\]
established in \cite{BFHilb,LEPA,JLI}, the relation \eqref{almostthere} becomes precisely
\[
\DDT_{\mathcal F}(q) = \mathsf M((-1)^rq)^{r\chi(X)} \cdot \PPT_{\mathcal F}(q).
\]
This completes the proof.
\end{proof}

We now prove Theorem \ref{thm3}, namely the main result of \cite{Gholampour2017} in the special case of a stable sheaf $\F$. We do not require $X$ to be Calabi--Yau.
\begin{theorem}\label{thm:gk}
Let $X$ be a smooth projective $3$-fold, and let $\mathcal F$ be as in Theorem \ref{thm:DTPT}. 
Then
\be \label{GK937}
\sum_{n\geq 0}\chi(\Quot_X(\mathcal F,n))q^n = \mathsf M(q)^{r\chi(X)}\cdot \sum_{n\geq 0}\chi(\Quot_X(\lExt^1(\mathcal F,\O_X),n))q^n.
\ee
\end{theorem}
\begin{proof}
The equality follows from the proof of \cite[Thm.~3.17]{Toda2} by replacing the Hall algebra identity of \cite[Lemma~3.16]{Toda2} by the identity of equation~\eqref{eq:Hall_Identity_Flocal} \emph{and} by replacing the Behrend weighted integration morphism $I^B$ by the integration morphism $I^E$ taking Euler characteristics.
The compatibility of $I^E$ with the Poisson brackets is explained in Remark~\ref{rem:Integration_Morphisms}.
\end{proof}
\begin{remark}
In \cite{Gholampour2017}, formula \eqref{GK937} is obtained by reducing to the affine case, and carrying out an inductive procedure on the rank.
The base case of rank two, established on an affine $3$-fold, relies on the existence of an auxiliary \emph{cosection} $\mathcal F\ra \O_X$.

The above result yields an interpretation of equation~\eqref{GK937} as the Euler characteristic shadow of the $\F$-local higher rank DT/PT correspondence \eqref{FlocalDTPT}. This question was raised by Gholampour and Kool in \cite[Sec.~1]{Gholampour2017}.

Note that our proof only produces the formula for $\mu_{\omega}$-stable sheaves $\F$ (of homological dimension at most one).
This is a consequence of producing the identity \eqref{eq:Hall_Identity_Flocal} in the Hall algebra of $\catA_{\mu}$.
In contrast, \cite{Gholampour2017} proves formula \eqref{GK937} for \emph{all} torsion free sheaves (of homological dimension at most one).
\end{remark}

\begin{remark}\label{rmk:cyclelocal}
Let $X$ be a Calabi--Yau $3$-fold, $C\subset X$ a Cohen--Macaulay curve, and $\F = \mathscr I_C$ of $\ch(\mathscr I_C) = (1,0,-\beta,p_a(C)-1)$.
The \emph{cycle-local} invariants of \cite[Sec.~4]{Ob1} may in general differ from those of Theorem \ref{thm:DTPT}.
Indeed, the former express the contribution of all ideal sheaves $\mathscr I_{Z}$ such that $[Z] = [C]$ in $\Chow_1(X,\beta)$, and this condition is in general weaker than having an inclusion $\mathscr I_Z\into \mathscr I_C$.
The two types of invariants do agree when $C$ is smooth: in this case, \cite[Thm. 2.1]{Ricolfi2018} proves that $\Quot_X(\mathscr I_C,n)$ is precisely the fibre of the Hilbert--Chow morphism over the cycle of $C$. They also agree when $\beta = \ch_2(\O_C)$ is irreducible, see Proposition \ref{prop:cycleVSframed}.

Also the method of proof differs: in \cite{Ob1}, the author restricts Bridgeland's global DT/PT identity to the Hall subalgebra of the abelian category of sheaves supported on $C$ in dimension one, whereas in this paper an $\mathscr I_C$-local identity is established directly in the Hall algebra of $\mathscr A_\mu$.
We emphasise, however, that both local proofs make use of the established proof of the global correspondence in some way. 
\end{remark}

\section{Applications}
In this section we discuss a few results and special cases, directly linked to Theorem~\ref{thm:DTPT}, concerning the PT series $\PPT_{\F}$ and its properties.
Throughout, as before, $X$ is a smooth projective 3-fold and $\mathcal{F}$ is a $\mu_{\omega}$-stable sheaf of homological dimension at most one.

\subsection{Tensoring by a line bundle}
We establish a general relation between the generating functions of $\mathcal F$-local and of $(\mathcal F \otimes L)$-local PT invariants, where $L$ is a line bundle on $X$.
\begin{prop}\label{prop:PT_tensor_by_line_bundle}
    We have $\PPT_{\mathcal F \otimes L}(q) = \PPT_{\mathcal F}(q)$ for every line bundle $L$ on $X$.
\end{prop}
\begin{proof}
Since $\mathcal F$ is a $\mu_{\omega}$-stable sheaf of homological dimension at most 1, the same holds for $\mathcal F \otimes L$. 
Moreover, $(r,D \cdot \omega^2)$ are coprime so $(r,(rc_1(L) + D)\cdot \omega^2)$ are coprime as well.

Let $t \colon \lExt^1(\mathcal F,\O_X) \onto Q$ be a zero-dimensional quotient.
Tensoring by $L^{-1}$ induces the identification
\[
\begin{tikzcd}
\lExt^1(\mathcal F,\O_X) \otimes L^{-1} \arrow[two heads]{r}\isoarrow{d} & Q \otimes L^{-1} \arrow[equal]{d} \\
\lExt^1(\mathcal F \otimes L,\O_X) \arrow[two heads]{r} & Q \otimes L^{-1} 
\end{tikzcd}
\]
using that Ext-sheaves are local; note that $Q$ is zero-dimensional so $Q \otimes L^{-1} \cong Q$, but for reasons of naturality we do not choose an isomorphism here.
Tensoring by line bundles behaves well in flat families, so we obtain an isomorphism
\[
- \otimes L^{-1} \colon \Quot_X(\lExt^1(\F,\O_X),n) \,\widetilde{\ra}\, \Quot_X(\lExt^1(\mathcal F \otimes L,\O_X),n).
\]

We claim that the following diagram
\begin{equation}\label{eq:Tensoring_PTmoduli}
\begin{tikzcd}[row sep=large]
\Quot_X(\lExt^1(\F,\O_X),n) \ar[hook,swap]{d}{\psi_{\mathcal F}} \ar{r}{- \otimes L^{-1}} & \Quot_X(\lExt^1(\mathcal F \otimes L,\O_X),n) \ar[hook]{d}{\psi_{\mathcal{F} \otimes L}} \\
M_{\PT}(r,D) \ar{r}{- \otimes L} & M_{\PT}(r,D+rc_1(L))
\end{tikzcd}
\end{equation}
commutes. Our closed immersion into the PT moduli space proceeds by dualising
\[
\bar{t} \colon \mathcal{F}^{\vee} \to \lExt^1(\mathcal F,\O_X)[-1] \to Q[-1].
\]
To obtain the corresponding $\F$-local PT pair, we dualise \emph{again} to obtain the extension
\[
\mathcal{F} \into J^{\bullet} \onto Q^{\textrm{D}}[-1],
\]
which corresponds to $\bar{t}$ under $\Hom(\mathcal F^{\vee},Q[-1]) \cong \Ext^1(Q^{\textrm{D}}[-1],\mathcal F )$.
Tensoring this exact sequence by the line bundle $L$ yields an $(\mathcal{F} \otimes L)$-local PT pair $J^{\bullet} \otimes L$, namely
\[
\mathcal{F} \otimes L \into J^{\bullet} \otimes L \onto Q^{\textrm{D}} \otimes L[-1].
\]

We also obtain a PT pair by performing the operations the other way around.
Indeed, first tensoring $\bar t$ by the line bundle $L^{-1}$ yields the morphism
\[
\bar{t}_L \colon \mathcal{F}^{\vee} \otimes L^{-1} \cong (\mathcal{F} \otimes L)^{\vee} \to \lExt^1(\mathcal F \otimes L,\O_X)[-1] \to (Q \otimes L^{-1})[-1].
\]
To obtain the corresponding local PT pair, we dualise \emph{again} to obtain the extension
\[
\mathcal{F} \otimes L \into J^{\bullet}_L \onto Q^{\textrm{D}} \otimes L[-1],
\]
corresponding to $\bar{t}_L$ under $\Hom((\mathcal F \otimes L)^{\vee},(Q \otimes L^{-1})[-1]) \cong \Ext^1(Q^{\textrm{D}} \otimes L[-1],\mathcal F \otimes L)$.
We claim that $J^{\bullet}_L$ and $J^{\bullet} \otimes L$ are canonically isomorphic $(\mathcal{F} \otimes L)$-PT pairs.

To see this, consider the following diagram of canonical and commuting isomorphisms:
\[
\begin{tikzcd}
\Hom(\lExt^1(\mathcal F,\O_X),Q) \ar[hook]{d} \ar[equal]{r}{-\otimes L^{-1}} & \Hom(\lExt^1(\mathcal F \otimes L,\O_X), Q \otimes L^{-1}) \ar[hook]{d} \\
\Hom(\mathcal F^{\vee},Q[-1]) \ar[equal,swap]{d}{(-)^{\vee}} \ar[equal]{r}{- \otimes L^{-1}} & \Hom((\mathcal F \otimes L)^{\vee}, Q \otimes L^{-1}[-1]) \ar[equal]{d}{(-)^{\vee}}  \\
\Ext^1(Q^{\textrm{D}}[-1],\mathcal F) \ar[equal]{r}{- \otimes L} & \Ext^1(Q^{\textrm{D}} \otimes L[-1], \mathcal F \otimes L)
\end{tikzcd}
\]
which proves that the operations indeed commute, so $J^{\bullet}_L \cong J^{\bullet} \otimes L$ canonically, because the extensions are equal.
Thus the diagram displayed in \eqref{eq:Tensoring_PTmoduli} commutes.

As a consequence, pulling back the Behrend functions of $M_{\PT}(r,D)$ and $M_{\PT}(r,D+rc_1(L))$ to either moduli space induces the \emph{same} constructible function on the isomorphic Quot schemes.
In particular, their Behrend weighted Euler characteristics are equal, which means that $\PPT_{\mathcal F,n} = \PPT_{\mathcal F \otimes L,n}$ for all $n \in \Z$.
We infer $\PPT_{\mathcal F}(q) = \PPT_{\mathcal F \otimes L}(q)$ as claimed.
\end{proof}

Consider the generating function of topological Euler characteristics
\[
\widehat{\PPT}_{\mathcal{F}}(q) = \sum_{n\geq 0}\chi(\Quot_X(\lExt^1(\mathcal F,\O_X),n))q^n.
\]
The above proof directly carries over to the Euler characteristic setting.
\begin{corollary}
We have $\widehat{\PPT}_{\mathcal{F} \otimes L}(q) = \widehat{\PPT}_{\mathcal{F}}(q)$ for every line bundle $L$ on $X$.
\end{corollary}

\subsection{Special cases}
Let $X$ be a smooth projective Calabi--Yau 3-fold.
We collect some special cases of Theorem~\ref{thm:DTPT} by imposing further restrictions on the sheaf $\mathcal{F}$.
\begin{corollary}\label{cor101}
With the assumptions of Theorem~\ref{thm:DTPT}, if $\F$ is locally free then the generating series of $\mathcal F$-local PT invariants is trivial: $\PPT_{\mathcal F}(q) = 1$.
In particular, 
\[
\DDT_{\F}(q) = \mathsf M((-1)^rq)^{r\chi(X)}.
\]
\end{corollary}
\begin{proof}
Since $\lExt^1(\F,\O_X) = 0$, one deduces from the definitions that $\PPT_{\F,n} = 0$ for all $n \neq 0$ and  that $\PPT_{\F,0} = 1$.
The formula for $\DDT_{\F}$ then follows from Theorem~\ref{thm:DTPT}.
\end{proof}
\begin{remark}
Recall that any rank one PT pair is of the form $I^\bullet = [s \colon \O_X \to F]$ where $F$ is a pure one-dimensional sheaf and $\coker(s)$ is zero-dimensional; in particular, $\ker(s) = H^0(I^\bullet)$ is the ideal sheaf of a Cohen--Macaulay curve $C \subset X$.
The above corollary generalises the fact that the only rank one PT pair $I^\bullet$ with $H^0(I^\bullet)$ a line bundle is the trivial one with $F = 0$.  
\end{remark}

\begin{remark}\label{rmk:comparison}
Combining Corollary \ref{cor101} with Theorem \ref{Thm:LocallyFree}, we observe that for a stable vector bundle $\F$ of rank $r$ on a Calabi--Yau $3$-fold $X$, one has the identity
\[
\DDT_{\F,n} = \widetilde\chi(\Quot_X(\F,n)).
\]
Note that the left hand side depends, a priori, on the embedding
\[
\phi_{\F,n}\colon\Quot_X(\F,n) \into M_{\DT}(r,D,\ch_2(\F),\ch_3(\F)-n)
\]
whereas the right hand side is completely intrinsic to the Quot scheme. The above identity is trivial in the case $(r,D)=(1,0)$, for then $\Quot_X(\F,n)=\Hilb^n X=M_{\DT}(1,0,0,-n)$.
\end{remark}

Next, we assume that $\mathcal F$ is a \emph{reflexive} $\mu_{\omega}$-stable sheaf. 
Let $(-)^*$ denote the usual $\O_X$-linear dual of a sheaf, and note that $\mathcal F^*$ is again $\mu_{\omega}$-stable and reflexive; see for example \cite[Lemma~2.1(2)]{Gholampour2017}. In particular, it is both a DT and PT object by Corollary~\ref{cor:DTPT_object} and the series $\PPT_{\mathcal F^*}(q)$ is well-defined.
We denote the \emph{reciprocal} of a polynomial $P(q)$ of degree $d$ by
\[
P^*(q) = q^dP(q^{-1}),
\]
and we let $\ell(T)$ denote the length of a zero-dimensional sheaf $T$.

Following its proof, we obtain the virtual analogue of \cite[Thm.~1.2]{Gholampour2017}.

\begin{corollary}\label{cor292}
With the assumptions of Theorem~\ref{thm:DTPT}, if $\mathcal F$ is reflexive then the series
\[
\dfrac{\DDT_{\mathcal F}(q)}{\mathsf M((-1)^rq)^{r\chi(X)}} = \PPT_{\mathcal F}(q)
\]
of $\mathcal F$-local PT invariants is a \emph{polynomial} of degree $\ell(\lExt^1(\mathcal F,\O_X))$.
Moreover, this polynomial has a symmetry induced by the derived dualising functor
\begin{equation}\label{eq:PTF_symmetry}
\PPT^*_{\mathcal F}(q) = \PPT_{\mathcal F^*}(q).
\end{equation}
And finally, if $\rk(\mathcal F) = 2$ then $\PPT^*_{\mathcal F}(q) = \PPT_{\mathcal F}(q)$ is palindromic.
\end{corollary}
\begin{proof}
The final claim requires the Behrend weighted identity $\PPT_{\mathcal F \otimes L}(q) = \PPT_{\mathcal F}(q)$ for any line bundle $L$, which is provided by Proposition~\ref{prop:PT_tensor_by_line_bundle}, and $\F^* \cong \F \otimes \det(\F)^{-1}$ if $\rk(\F) = 2$.
\end{proof}

\subsection{Rationality: open questions}
Let $\beta\in H_2(X,\Z)$ be a curve class on a Calabi--Yau $3$-fold $X$, and let $\PPT_{n,\beta}$ be the (rank one) PT invariant, defined as the degree of the virtual fundamental class of $P_n(X,\beta)=M_{\PT}(1,0,-\beta,-n)$. In \cite[Conj.~3.2]{PT}, Pandharipande and Thomas conjectured that the Laurent series
\[
\PPT_\beta (q) = \sum_{n\in \Z}\PPT_{n,\beta} q^n
\]
is the expansion of a \emph{rational function} in $q$ (invariant under $q\leftrightarrow q^{-1}$).
This was proved by Bridgeland \cite{Bri}. More generally, Toda \cite[Thm.~1.3]{Toda2} proved the  rationality of the series
\[
\PPT_{r,D,\beta}(q)=\sum_{6n\in \Z}\PPT(r,D,-\beta,-n) q^n
\]
for arbitrary $(r,D,\beta)$. Moreover, Toda previously proved \cite{Toda0} the rationality of the unweighted generating function
\[
\widehat{\PPT}_{\beta}(q) = \sum_{n\in \Z}\chi(P_n(X,\beta)) q^n.
\]
One may ask similar questions about the local invariants studied in the present paper.

Let $\F$ be a $\mu_{\omega}$-stable sheaf of homological dimension at most one on the Calabi--Yau $3$-fold $X$. It makes sense to ask the following:

\begin{question}
On a Calabi--Yau $3$-fold $X$, is $\PPT_{\F}(q)$ the expansion of a rational function?
\end{question}

One can of course ask the same question for the unweighted invariants
\[
\widehat{\PPT}_{\F}(q)=\sum_{n\geq 0}\chi(\Quot_X(\lExt^1(\F,\O_X),n)) q^n,
\]
where $X$ is now an arbitrary smooth projective $3$-fold and $\F$ is a torsion free sheaf of homological dimension at most one (not necessarily stable).

\begin{question}
On a $3$-fold $X$, is $\widehat{\PPT}_{\F}(q)$ the expansion of a rational function?
\end{question}
Rationality of $\widehat{\PPT}_{\F}$ has been announced for toric $3$-folds in \cite{Gholampour2017}.
As for the weighted version, there is a partial answer in the rank one case, building upon work of Pandharipande and Thomas \cite{BPS}.
We will investigate such rationality questions in future work.

\begin{prop}\label{prop:cycleVSframed}
Let $X$ be a Calabi--Yau $3$-fold, $C\subset X$ a Cohen--Macaulay curve in class $\beta\in H_2(X,\Z)$. If $\beta$ is irreducible, then $\PPT_{\mathscr I_C}(q)$ is the expansion of a rational function in $q$.
\end{prop}

\begin{proof}
Let $g=1-\chi(\O_C)$ be the arithmetic genus of $C$, and denote by $P_n(X,C)\subset P_{1-g+n}(X,\beta)$ the closed subset parametrising stable pairs $[\O_X\ra F]$ such that the fundamental one-cycle of $F$ equals $[C]\in \Chow_1(X,\beta)$. Let $\mathsf P_{n,C}$ denote the virtual contribution of $P_n(X,C)$. Under the irreducibility assumption on $\beta$, Pandharipande and Thomas showed in \cite[Sec.~3.1]{BPS} that the generating function of cycle-local invariants
\[
\mathsf Z_C(q) = \sum_{n\geq 0} \mathsf P_{n,C}q^{1-g+n}
\]
admits the unique expression
\[
\mathsf Z_C(q) = \sum_{r=0}^g\mathsf n_{r,C}q^{1-r}(1+q)^{2r-2}
\]
as a rational function of $q$, where $\mathsf n_{r,C}$ are integers, called the ``BPS numbers'' of $C$. Since $\beta$ is irreducible, however, we have
\[
\mathsf P_{n,C} = \PPT_{\mathscr I_C,n}.
\]
Indeed, the Chow variety parametrises Cohen--Macaulay curves on $X$ in class $\beta$, therefore given a stable pair $[s \colon \O_X\ra F]\in P_n(X,C)$ along with its induced short exact sequence
\[
0\ra \O_X/\ker s\ra F \ra Q \ra 0,
\]
the condition $[F] = [C] \in \Chow_1(X,\beta)$ implies the identity $\ker s = \mathscr I_C$. It follows that $\PPT_{\mathscr I_C} = \mathsf Z_C$, whence the result.
\end{proof}

\appendix
\section{\'{E}tale maps between Quot schemes}\label{quotmess}
Let $\varphi\colon X\ra X'$ be a morphism of varieties, with $X'$ proper. Let $F'$ be a coherent sheaf on $X'$, set $F = \varphi^\ast F'$ and
\be\label{quots123}
Q = \Quot_X(F,n),\qquad Q' = \Quot_{X'}(F',n).
\ee
If we have a scheme $S$, we denote by $F_S$ the pullback of $F$ along the projection $X\times S\ra X$. For instance, we let
\[
F_Q\surj \mathscr T
\]
denote the universal quotient, living over $X\times Q$.

\begin{lemma}\label{lemma:surj726}
Let $\varphi\colon X\ra X'$ be an \'etale map of quasi-projective varieties, $F'$ a coherent sheaf on $X'$ and let $F = \varphi^\ast F'$. Let $[\theta\colon F\surj T]\in Q$ be a point such that $\varphi$ is injective on $\Supp T$. Then
\be\label{surj837}
F'\ra \varphi_\ast\varphi^\ast F'\ra \varphi_\ast T
\ee
stays surjective.
\end{lemma}

\begin{proof}
First of all, to check surjectivity of \eqref{surj837}, we may replace $X'$ by any open neighborhood of the support of $\varphi_\ast T$, for example the image of $\varphi$ itself. So now $\varphi$ is faithfully flat, hence $\varphi^\ast$ is a faithful functor, in particular it reflects epimorphisms. It is easy to see that we may also replace $X$ with any open neighborhood $V$ of the support $B = \Supp T$.
Since $\varphi$ is \'etale, it is in particular quasi-finite and unramified: the fibres $X_p$, for $p\in X'$, are finite, reduced of the same length. Then
\[
A=\coprod_{b\in B}X_{\varphi(b)}\setminus \Set{b}\subset X
\]
is a closed subset, and we can consider the open neighborhood
\[
B\subset V = X\setminus A\subset X
\]
of the support of $T$. Since $\varphi$ is injective on $B$, and we have just removed the points in $\varphi^{-1}(\varphi(B))$ that are not in $B$, the map $\varphi$ is now an immersion around $B$, so after replacing $X$ by $V$ we observe that the canonical map
\be\label{can191}
\varphi^\ast\varphi_\ast T\ra T
\ee
is an isomorphism.
Let us pullback \eqref{surj837} along $\varphi$, to get
\be\label{pullbackl173}
\rho\colon F\ra \varphi^\ast\varphi_\ast T.
\ee
If we compose $\rho$ with \eqref{can191} we get back $\theta\colon F\surj T$, the original surjection. But \eqref{can191} is an isomorphism, thus $\rho$ is surjective. Since $\varphi^\ast$ reflects epimorphisms, \eqref{surj837} is also surjective, as claimed.
\end{proof}

\begin{prop}\label{prop:comparison}
Let $\varphi\colon X\ra X'$ be a morphism of varieties, with $X'$ proper. Let $Q$ and $Q'$ be Quot schemes as in \eqref{quots123}. 
Fix a point $\theta = [F\surj T]\in Q$ such that $\varphi$ is \'etale around $B = \Supp T$ and $\varphi|_{B}$ is injective. Then there is an open neighborhood $\theta\in U\subset Q$ admitting an \'etale map $\Phi\colon U\ra Q'$.
\end{prop}

\begin{proof}
First of all, we may replace $X$ by any open neighborhood $V$ of $B$. We may choose $V$ affine, so we may assume $\varphi$ is affine and \'etale. In the diagram
\[
\begin{tikzcd}
X\times \theta \MySymb{dr} \arrow[hook]{r}{i} \arrow[swap]{d}{\varphi} 
& X\times Q \MySymb{dr} \arrow{r}{p}\arrow{d}{\widetilde\varphi} & X\arrow{d}{\varphi}\\
X'\times \theta \arrow[swap,hook]{r}{j} 
& X'\times Q\arrow[swap]{r}{\rho} & X'
\end{tikzcd}
\]
the map $\widetilde\varphi$ is now affine, so $j^\ast \widetilde\varphi_\ast\,\widetilde{\ra}\,\varphi_\ast i^\ast$, and similarly we have $\rho^\ast \varphi_\ast\,\widetilde{\ra}\,\widetilde\varphi_\ast p^\ast$ by flat base change. Let us look at the canonical map
\[
\alpha\colon \rho^\ast F'\ra \rho^\ast \varphi_\ast F\,\widetilde{\ra}\, \widetilde\varphi_\ast F_Q\surj \widetilde\varphi_\ast \mathscr T.
\]
We know by Lemma \ref{lemma:surj726} that restricting $\alpha$ to $\theta\in Q$ (that is, applying $j^\ast$) we get a surjection $F'\ra \varphi_\ast F\surj \varphi_\ast T$. Since $\varphi$ is \'etale and $\varphi|_B$ is injective, this gives a well-defined point
\[
\varphi_\ast\theta = [F'\surj \varphi_\ast T] \in Q'.
\]
Now we extend the association $\theta\mapsto \varphi_\ast\theta$ to a morphism $\Phi\colon U\ra Q'$ for suitable $U\subset Q$.
Note that $\widetilde\varphi_\ast \mathscr T$, the target of $\alpha$, is coherent (reason: $\Supp \mathscr T\ra Q$ proper and factors through the separated projection $X'\times Q\ra Q$, so that $\widetilde\varphi\colon \Supp \mathscr T\ra X'\times Q$ is proper; but $\mathscr T$ is the pushforward of a coherent sheaf on its support). Then the cokernel $\mathscr K$ of $\alpha$ is also coherent, so that $\Supp \mathscr K\subset X'\times Q$ is closed. Since $X'$ is proper, we have that the projection $\pi\colon X'\times Q\ra Q$ is closed, so the image of the support of $\mathscr K$ is closed. Let 
\[
U = Q\setminus \pi(\Supp \mathscr K)\subset Q
\]
be the open complement (non-empty because $\theta$ belongs there by assumption).
Now consider the cartesian square
\[
\begin{tikzcd}
X\times U \MySymb{dr} \arrow[hook]{r}{} \arrow[swap]{d}{\varphi_U} 
& X\times Q\arrow{d}{\widetilde\varphi} \\
X'\times U \arrow[swap,hook]{r}{} 
& X'\times Q
\end{tikzcd}
\]
and observe that by construction, $\alpha$ restricts to a surjection
\[
\alpha\big|_{X'\times U}\colon \rho_U^\ast F' \surj \varphi_{U\ast}\mathscr T\big|_{X\times U}
\]
where $\rho_U$ is the projection $X'\times U\ra X'$. But the target $\varphi_{U\ast}\mathscr T\big|_{X\times U}$ is flat over $U$ (reason: $\widetilde\varphi_\ast\mathscr T$ is flat over $Q$ because $\mathscr T$ is; but $\varphi_{U\ast}\mathscr T\big|_{X\times U}$ is isomorphic to the pullback of $\widetilde\varphi_\ast\mathscr T$ along the open immersion $X'\times U\ra X'\times Q$, therefore it is flat over $U$). We have constructed a morphism
\[
\Phi\colon U\ra Q'.
\]
Now we show it is \'etale by using the infinitesimal criterion. First of all, as in the proof of Lemma \ref{lemma:surj726}, we may shrink $X$ further in such a way that, for every $x\in B$, the fibre $X_{\varphi(x)}$ consists of the single point $x$. This implies that the canonical map $\varphi^\ast\varphi_\ast T\ra T$ is an isomorphism; moreover, this condition only depends on the set-theoretic support of $T$, so it is preserved in infinitesimal neighborhoods; in particular we have
\be\label{fat}
\varphi^\ast\varphi_\ast\mathscr F\,\widetilde{\ra}\,\mathscr F
\ee
for all infinitesimal deformations $F_S\surj \mathscr F$ of $\theta$ parametrized by a fat point $S$.

Let $\iota\colon S\ra \overline S$ be a square zero extension of fat points, and consider a commutative diagram
\[
\begin{tikzcd}
S\arrow[hook,swap]{d}{\iota}\arrow{r}{g} & U\arrow{d}{\Phi} \\
\overline S\arrow[swap]{r}{h}\arrow[dotted]{ur}[description]{v} & Q'
\end{tikzcd}
\]
where we need to find a unique $v$ making the two triangles commutative. This will correspond to a family
\[
F_{\overline S}\surj \mathscr V
\]
that we have to find.
To fix notation, consider the fibre diagram
\be\label{fibre618}
\begin{tikzcd}
X\times S \MySymb{dr} \arrow[hook]{r}{\iota_X}\arrow[swap]{d}{\varphi_S} & X\times \overline S\arrow{d}{\varphi_{\overline S}} \\
X'\times  S\arrow[hook,swap]{r}{\iota_{X'}} & X'\times \overline S
\end{tikzcd}
\ee
and denote by $F_S\surj \mathscr G$ the family corresponding to $g$ (restricting to $\theta$ over the closed point) and by
$F'_{\overline S}\surj \mathscr H$ is the family corresponding to $h$. The condition $\varphi\circ g = h\circ \iota$ means that we have a diagram of sheaves
\be\label{diag22232}
\begin{tikzcd}
\iota_{X'}^\ast F'_{\overline S} \arrow[equal]{d} \arrow[two heads]{r} & \iota_{X'}^\ast\mathscr H\isoarrow{d} \\
F'_S \arrow[two heads]{r} & \varphi_{S\ast}\mathscr G
\end{tikzcd}
\ee
over $X'\times S$. The conditions $\varphi_{\overline S\ast}\mathscr V=\mathscr H$ (translating $\varphi\circ v = h$) and 
$\varphi^\ast_{\overline S}\varphi_{\overline S\ast}\mathscr V = \mathscr V$, coming from \eqref{fat}, together determine for us the family
\[
\mathscr V = \varphi_{\overline S}^\ast \mathscr H,
\]
which we consider together with the natural surjection $\varphi^\ast(h)\colon F_S\surj \mathscr V$. Note that 
\begin{align*}
\iota_X^\ast\mathscr V 
&= \iota_X^\ast \varphi_{\overline S}^\ast\mathscr H & \\
&= \varphi_S^\ast \iota_{X'}^\ast \mathscr H & \textrm{by }\eqref{fibre618} \\
&= \varphi_S^\ast \varphi_{S\ast}\mathscr G & \textrm{by } \eqref{diag22232} \\
&= \mathscr G & \textrm{by } \eqref{fat}
\end{align*}
proves that $v\circ \iota = g$, finishing the proof.
\end{proof}

Finally, we extend the statement to $X'$ quasi-projective.

\begin{prop}\label{prop:etaleV}
Let $\varphi\colon X\ra X'$ be an \'etale map of quasi-projective varieties, $F'$ a coherent sheaf on $X'$, and let $F = \varphi^\ast F'$. Let $V\subset Q$ be the open subset whose points correspond to quotients $F\surj T$ such that $\varphi|_{\Supp T}$ is injective. Then there is an \'etale morphism $\Phi\colon V\ra Q'$.
\end{prop}

\begin{proof}
First complete $X'$ to a proper scheme $Y$. Let $i\colon X'\ra Y$ be the open immersion and note that $i^\ast i_\ast F' = F'$ canonically. Combining Proposition \ref{prop:comparison} and Lemma \ref{lemma:surj726} we get an \'etale map $\Phi\colon V\ra \Quot_{Y}(i_\ast F',n)$. But the support of any quotient sheaf $i_\ast \varphi_\ast T$ lies in the open part $X'\subset Y$, so $\Phi$ actually factors through $\Quot_{X'}(F',n) = Q'$.
\end{proof}


\section{Relative Quot schemes and PT pairs}\label{UniversalEmbedding}
Let $X$ be a Calabi--Yau $3$-fold and fix the Chern character $\alpha = (r,D,-\beta,-m)$ with $r \geq 1$.
Let $\omega\in H^2(X,\Z)$ be an ample class on $X$ satisfying the usual coprimality condition
\be\label{coprimeagain}
\gcd(r,D\cdot \omega^2) = 1.
\ee

In this section, we upgrade the `fibrewise' closed immersion
\[
\psi_{\F,n}\colon \Quot_X(\lExt^1(\F,\O_X),n)\into M_{\PT}(r,D,\ch_2(\F),\ch_3(\F)-n)
\]
of Proposition \ref{PTembedding}, where $\F$ is a DT and PT object of $\ch(\F) = \alpha$, to a \emph{universal} closed immersion $\psi_{\alpha,n}$ by allowing the sheaf $\F$ to vary in the moduli space of DT and PT objects.
The domain of $\psi_{\alpha,n}$ is a certain relative Quot scheme.
The morphism $\psi_{\alpha,n}$ is not an isomorphism; indeed, it even fails to be surjective on $\C$-valued points. However, the coproduct $\psi_{r,D,\beta} = \coprod_{n \in \Z} \psi_{\alpha,n}$ (not depending on $\ch_3$ anymore) is a geometric bijection onto
\be\label{eq:MPT_rDbeta}
M_{\PT}(r,D,-\beta) = \coprod_{n \in \Z} M_{\PT}(r,D,-\beta,-m-n).
\ee
This may be seen as a new stratification of $M_{\PT}(r,D,-\beta)$ by relative Quot schemes. 

We first show that there exists a universal sheaf on $X\times M_{\DT}(\alpha)$. 
According to \cite[\textsection~4.6]{modulisheaves}, quasi-universal families exist on every moduli space of stable sheaves on a smooth projective variety.
A universal family exists if the following numerical criterion is fulfilled.
\begin{theorem}[{\cite[Thm.~4.6.5]{modulisheaves}}]\label{thm:HLcriterion}
Let $Z$ be a smooth projective variety, let $\mathsf{c}$ be a class in the numerical Grothendieck group $N(Z)$, and let $\set{B_1,\ldots,B_\ell}$ be a collection of coherent sheaves.
If
\[
\gcd(\chi(\mathsf{c} \otimes B_1), \ldots, \chi(\mathsf{c} \otimes B_\ell)) = 1,
\]
then there exists a universal family on $M(\mathsf{c})^{s} \times Z$.
\end{theorem}

We obtain the following corollary.
\begin{corollary}\label{cor98274}
There exists a universal family $\mathscr{F}_{\DT}(\alpha)$ on $M_{\DT}(\alpha) \times X$.
\end{corollary}
Note that $c_1 \colon \Pic(X) \to H^2(X,\Z)$ is an isomorphism since $H^i(X,\O_X) = 0$ for $i=1,2$. We let $L$ be a line bundle represented by $\omega\in H^2(X,\Z)$. Since $X$ is Calabi--Yau, its Todd class is
\[
\Td(X) = (1,0,c_2(X)/12,0).
\]

\begin{proof}
We apply the above criterion.
Take the sheaves $B_k = L^{\otimes k}$ for $k = 0,1,2$ and $B_3 = \O_x$, where $x\in X$ is a point.
Via the Hirzebruch--Riemann--Roch Theorem, and using the above expression for the Todd class of $X$, we compute
\begin{equation*}
    \chi(\alpha \otimes L^{\otimes k}) 
    =\frac{rk^3}{6}\omega^3 + \frac{k^2}{2}D\cdot \omega^2 + \frac{rk}{12}c_2(X)\cdot \omega - k \beta\cdot \omega + \frac{1}{12}D\cdot c_2(X) - m
\end{equation*}
and $\chi(\alpha \otimes \O_x) = r$.
Writing $P_{\alpha}(k) = \chi(\alpha \otimes L^{\otimes k})$, it is not hard to see that
\[
P_{\alpha}(2) - 2P_{\alpha}(1) + P_{\alpha}(0) = r \omega^3 + D \cdot \omega^2.
\]
By the coprimality assumption \eqref{coprimeagain}, we find
\begin{equation*}
\begin{split}
    \gcd(\chi(\alpha \otimes \O_x),P_{\alpha}(2),P_{\alpha}(1),P_{\alpha}(0))
    &= \gcd(r,P_{\alpha}(2)-2P_{\alpha}(1)+P_{\alpha}(0),P_{\alpha}(1),P_{\alpha}(0)) \\
    &= \gcd(r,r \omega^3 + D \cdot \omega^2,P_{\alpha}(1),P_{\alpha}(0)) \\
    &=1.
\end{split}
\end{equation*}
The result now follows from Theorem \ref{thm:HLcriterion}.
\end{proof}

Next, we describe the domain of the morphism $\psi_{\alpha,n}$, which is a relative Quot scheme.
Inside the $\alpha$-component of Lieblich's moduli stack $\mathscr M_X(\alpha)\subset \mathscr M_X$, consider the intersection
\[
\mathcal M(\alpha)=\mathcal M_{\DT}(\alpha)\cap \mathcal M_{\PT}(\alpha),
\]
and let $M(\alpha)$ be its coarse moduli space, consisting of $\mu_\omega$-stable sheaves of homological dimension at most one.
Let 
\[
\mathscr F_\alpha = \mathscr{F}_{\DT}(\alpha)\big|_{X\times M(\alpha)}
\]
be the restriction of the universal sheaf constructed in Corollary \ref{cor98274}; it is flat over $M(\alpha)$.
If $\iota_m\colon X\times\set{m}\into X\times M(\alpha)$ denotes the natural closed immersion, it follows that
\be\label{Vanish321}
\lExt^2\left(\iota_m^\ast\mathscr F_\alpha,\O_{X}\right) = 0
\ee
for every $m\in M(\alpha)$.
Then \cite[Thm.~1.10]{PicScheme} implies that $\lExt^2\left(\mathscr F_\alpha,\O_{X\times M(\alpha)}\right) = 0$ (using that $\mathscr F_\alpha$ is $M(\alpha)$-flat) and \cite[Thm.~1.9]{PicScheme} implies that in the previous degree the base change map to the fibre is an isomorphism
\be\label{BCiso}
\iota_m^\ast\lExt^1\left(\mathscr F_\alpha,\O_{X\times M(\alpha)}\right)\,\widetilde{\ra}\,\lExt^1\left(\iota_m^\ast\mathscr F_\alpha,\O_{X}\right).
\ee

We consider the relative Quot scheme
\[
Q(\alpha,n) \defeq \Quot_{X\times M(\alpha)}\left(\lExt^1\left(\mathscr F_\alpha,\O_{X\times M(\alpha)}\right),n\right).
\]
It is a projective $M(\alpha)$-scheme that represents the functor $(\Sch/M(\alpha))^{\mathrm{op}}\ra\Sets$ which sends a morphism $f\colon S\ra M(\alpha)$ to the set of equivalence classes of surjections
\[
\pi_f^\ast \lExt^1\left(\mathscr F_\alpha,\O_{X\times M(\alpha)}\right)\onto \mathscr Q.
\]
Here $\mathscr Q$ is an $S$-flat family of length $n$ sheaves, $\pi_f = \id_X \times f$ fits in the cartesian square
\[
\begin{tikzcd}
X \times S\MySymb{dr} \arrow{r}{\pi_f}\arrow[swap]{d}{\pr_2} & X\times M(\alpha) \arrow{d}{\pr_2} \\
S\arrow[swap]{r}{f} & M(\alpha)
\end{tikzcd}
\]
and two surjections are equivalent if they have the same kernel.

If $m = f(s)$, and $i_s\colon X\times\set{s}\into X\times S$ and $\iota_m\colon X\times\set{m}\into X\times M(\alpha)$ denote the natural closed immersions, then we have a canonical identification
\be\label{idFibers}
\begin{tikzcd}
X\times\set{s}\arrow[hook]{d}{i_s}\arrow{r}{\sim} & X\times \set{m}\arrow[hook]{d}{\iota_m} \\
X\times S\arrow{r}{\pi_f} & X\times M(\alpha)
\end{tikzcd}
\ee
which we rephrase as the condition
\be\label{composition2874}
\iota_m = \pi_f\circ i_s. 
\ee
\begin{remark}
Since $\mathscr{F}_{\alpha}\in \Coh(X \times M(\alpha))$ is flat over $M(\alpha)$, we have that ${\textbf L} \iota_m^*\mathscr{F}_{\alpha}$ is a coherent sheaf on $X$ for all $m \in M(\alpha)$.
In particular, ${\textbf L}^k \iota_m^*\mathscr{F}_{\alpha} = 0$ for all $k > 0$, and so ${\textbf L} \iota_m^*\mathscr{F}_{\alpha} = \iota_m^*\mathscr{F}_{\alpha}$.
Using \eqref{composition2874}, we obtain
\[
\iota_m^*\mathscr{F}_{\alpha} = {\textbf L}\iota_m^*\mathscr{F}_{\alpha} \cong {\textbf L}i_s^* \bigl({\textbf L}\pi_f^*\mathscr{F}_{\alpha}\bigr).
\]
By \cite[Lem.~4.3]{Bri2}, we conversely deduce that ${\textbf L}\pi_f^*\mathscr F_{\alpha}$ is a sheaf on $X \times S$ flat over $S$.
In particular, $\pi_f^*\mathscr F_{\alpha} = {\textbf L} \pi_f^*\mathscr F_{\alpha}$ since $\mathscr F_{\alpha}$ is a sheaf and hence ${\textbf L} \pi_f^*\mathscr F_{\alpha} \in D^{\leq 0}(X \times S)$.
\end{remark}

\begin{remark}\label{rem:family_version_constant_family}
If the morphism $f \colon S \to M(\alpha)$ is constant, corresponding to a single sheaf $\F$ parametrised by $S$ in a constant family, then $\pi_f$ factors through the flat projection
\[
X \times S \to X \times \{\mathcal{F}\} \subset X \times M(\alpha).
\]
We are reduced to the situation of Proposition~\ref{PTembedding} and $Q(\alpha,n)(f) = \Quot_X(\lExt^1_X(\F,\O_X),n)(S)$. 
\end{remark}

We now prove the following generalisation of Proposition \ref{PTembedding}.
\begin{prop}\label{prop:UnivPTimmersion}
For every $n\geq 0$ there is a closed immersion
\[
\psi_{\alpha,n}\colon Q(\alpha,n)\into M_{\PT}(r,D,-\beta,-m-n).
\]
\end{prop}

\begin{proof}
Fix a morphism $f\colon S\ra M(\alpha)$ and a quotient 
\[
q\colon \pi_f^\ast \lExt^1\left(\mathscr F_\alpha,\O_{X\times M(\alpha)}\right)\onto \mathscr Q.
\]
Note that the restriction of $q$ to the slice $X\times \set{s}$ is canonically identified with a surjection
\be\label{map:qs}
q_s\colon \lExt^1(\iota_m^\ast\mathscr F_\alpha,\O_X)\onto \mathscr Q_s
\ee
via the identification \eqref{idFibers} and the base change isomorphism \eqref{BCiso}.

We claim that there exists a canonical morphism
\[
u \colon \left(\pi_f^*\mathscr{F}_{\alpha}\right)^{\vee} \to \pi_f^*\lExt^1\left(\mathscr{F}_{\alpha},\O_{X \times M(\alpha)}\right)[-1].
\]
Indeed, it is the composition of the canonical morphisms
\[
\left(\pi_f^*\mathscr{F}_{\alpha}\right)^{\vee} = \bigl({\textbf L} \pi_f^*\mathscr{F}_{\alpha}\bigr)^{\vee} \cong {\textbf L} \pi_f^* \mathscr{F}_{\alpha}^{\vee} \to \pi_f^*\lExt^1\left(\mathscr F_{\alpha},\O_{X \times M(\alpha)}\right)[-1].
\]
The second isomorphism follows because the derived pullback commutes with the derived dualising functor.
The last morphism exists by the following argument.

The vanishing of \eqref{Vanish321} implies $\mathscr{F}_{\alpha} \in D^{[0,1]}(X \times M(\alpha))$.
There is a canonical triangle
\[
{\textbf L} \pi_f^*H^0(\mathscr{F}_{\alpha}^{\vee}) \to {\textbf L} \pi_f^*\mathscr{F}_{\alpha}^{\vee} \to {\textbf L} \pi_f^*\lExt^1\left(\mathscr{F}_{\alpha},\O_{X \times M(\alpha)}\right)[-1]
\]
since ${\textbf L} \pi_f^*$ is an exact functor.
Taking cohomology yields the isomorphism
\[
H^1\bigl({\textbf L} \pi_f^*\mathscr{F}_{\alpha}^{\vee}\bigr) \cong \pi_f^*\lExt^1\left(\mathscr{F}_{\alpha},\O_{X \times M(\alpha)}\right)
\]
since ${\textbf L} \pi_f^*$ is a left derived functor and so ${\textbf L}\pi_f^*\F_{\alpha} \in D^{\leq 1}(X \times S)$.
This completes the argument.

We now define $\psi_{\alpha,n}$.
As before, we shift the morphism $q$ and precompose by $u$ to obtain
\[
\bar{q} = q[-1] \circ u \colon \left(\pi_f^* \mathscr{F}_{\alpha}\right)^{\vee} \to \mathscr{Q}[-1].
\]
Taking cohomology of the canonical isomorphism $\mathbf R\lHom(E,F) \cong \mathbf R\lHom(F^{\vee},E^{\vee})$ yields
\[
\Hom_{X \times S}\bigl((\pi_f^* \mathscr{F}_{\alpha})^{\vee}, \mathscr{Q}[-1]\bigr) \cong \Ext^1_{X \times S}\bigl(\mathscr{Q}^{\mathrm D}[-1],\pi_f^* \mathscr{F}_{\alpha}\bigr),
\]
where we have appealed to Lemma \ref{TopExt}, just as in the proof of Proposition \ref{PTembedding}.
We identify $\bar{q}$ with the corresponding extension under this identification, to obtain a triangle
\begin{equation}\label{JSfamily}
\pi_f^*\mathscr{F}_{\alpha} \to J^{\bullet} \to \mathscr{Q}^{\mathrm D}[-1]
\end{equation}
of perfect complexes on $X \times S$.
We claim that $J^{\bullet}$ defines an $S$-family of PT pairs.

Clearly, the derived fibre $J^{\bullet}_s = {\textbf L} i_s^*J^{\bullet}$ of each closed point $s \in S$ defines a pre-PT pair based at the sheaf corresponding to $m = f(s)\in M(\alpha)$.
To see that $J^{\bullet}_s$ is a PT pair, recall that for a perfect complex the operations of taking the derived fibre and taking the derived dual commute.
By taking the derived fibre of the triangle \eqref{JSfamily}, we obtain the triangle
\[
{\textbf L}i_s^* \pi_f^*\mathscr{F}_{\alpha} \to J^{\bullet}_s \to \mathscr{Q}^{\mathrm D}_s[-1]
\]
in $\Perf(X)$.
Using the fact that ${\textbf L}i_s^*\pi_f^*\mathscr F_{\alpha} = \iota_m^*\mathscr F_{\alpha}$ and applying the derived dualising functor ${\textbf R}\lHom_X(-,\O_X)$, we obtain the triangle
\[
\mathscr{Q}_s[-2] \to J_s^{\bullet \vee} \to \bigl(\iota_m^*\mathscr{F}_{\alpha}\bigr)^{\vee}.
\]
Further taking cohomology yields the exact sequence
\[
\cdots \to \lExt^1_{X_s}\bigl(\iota_m^*\mathscr{F}_{\alpha},\O_{X_s}\bigr) \xrightarrow{q_s} \mathscr{Q}_s \to \lExt^2_{X_s}(J^{\bullet}_s,\O_{X_s}) \to 0
\]
where the last $0$ is $\lExt^2(\iota_m^\ast\mathscr F_\alpha,\O_{X})$ as in \eqref{Vanish321} and $q_s$ is precisely the map \eqref{map:qs} obtained by restricting the original surjection $q$. As such, it is surjective, therefore $\lExt^2(J^{\bullet}_s,\O_{X_s}) = 0$, proving that $J^\bullet$ defines a family of PT pairs by Lemma~\ref{lem:PrePT_is_PT}. 
We conclude that 
\[
\psi_{\alpha,n}\colon q\mapsto J^\bullet
\]
defines a morphism. The properness of $\psi_{\alpha,n}$ follows from the valuative criterion, along with the fact that $\psi_{\alpha,n}$ restricted to a fibre $\tau^{-1}(\F)$ of the structure morphism $\tau\colon Q(\alpha,n)\ra M(\alpha)$ is precisely the closed immersion $\psi_{\F,n}$ of Proposition \ref{PTembedding}; see Remark~\ref{rem:family_version_constant_family}.
Finally, the same argument used in Proposition \ref{PTembedding} shows injectivity on all valued points, thus proving that $\psi_{\alpha,n}$ is a closed immersion as claimed.
\end{proof}

Set $Q(r,D,-\beta) = \coprod_{n \in \Z} Q(\alpha,n)$ and recall the morphism $\psi_{r,D,\beta}$ from equation~\eqref{eq:MPT_rDbeta}.
\begin{corollary}
The morphism $\psi_{r,D,\beta} \colon Q(r,D,-\beta) \to M_{\PT}(r,D,-\beta)$ is a geometric bijection.
\end{corollary}
This may be seen as a new stratification of $M_{\PT}(r,D,-\beta)$ by relative Quot schemes.
\begin{proof}
Let $\F \to J^{\bullet} \to Q^{\mathrm D}[-1]$ be a $\C$-valued point of $M_{\PT}(r,D,-\beta)$.
Write $\ch_3(\F) = -m$ and $\ell(Q^{\mathrm D}) = n$, so that $\ch_3(J^{\bullet}) = -m-n$.
Set $\alpha = (r,D,-\beta,-m)$ and let $f \colon \Spec(\C) \to M(\alpha)$ be the morphism corresponding to $\F$.
Applying $\lHom(-,\O_X)$ to the triangle of $J^{\bullet}$ yields a surjection $q \colon \lExt^1_X(\F,\O_X) \onto Q$, which is a $\C$-valued point of $Q(\alpha,n)$.
\end{proof}

\bibliographystyle{amsplain}
\bibliography{bib}

\end{document}